\documentclass[11pt]{amsart}

\usepackage{amsmath,amssymb,amscd,amsthm,indentfirst,eufrak}
\usepackage{amsfonts}
\usepackage[all]{xy}
\usepackage{graphicx}
\usepackage{stmaryrd}
\usepackage{enumitem}
\usepackage{hyperref}

\newtheorem{prop}{Proposition}[section]
\newtheorem{theorem}[prop]{Theorem}
\newtheorem{cor}[prop]{Corollary}

\newtheorem{lemma}[prop]{Lemma}

\theoremstyle{definition}
\newtheorem{defn}[prop]{Definition}

\newtheorem{void}[prop]{}

\def\vp{\varphi}

\def\A{\mathcal{A}}
\def\B{\mathcal{B}}
\def\C{\mathcal{C}}
\def\D{\mathcal{D}}

\def\F{\mathcal{F}}

\def\H{\mathcal{H}}

\def\O{\ensuremath{\mathcal{O}}}

\def\U{\ensuremath{\mathcal{U}}}

\def\PP{\ensuremath{\mathbb{P}}}
\def\RR{\ensuremath{\mathbb{R}}}

\def\spaces{\textbf{Spaces}}

\def\catsq{\mathfrak{S}}

\def\cgh{\mathcal{CGH}}
\def\cgwh{\mathcal{CGWH}}

\def\colim{\operatorname{colim}}
\def\conn{\operatorname{conn}}

\def\ev{\operatorname{ev}}

\def\Fib{\operatorname{Fib}}
\def\fib{\operatorname{fib}}

\def\hofib{\operatorname{hoFib}}
\def\holim{\operatorname{holim}}

\def\id{\operatorname{id}}
\def\incl{\operatorname{incl}}
\def\Ind{\operatorname{Ind}}
\def\Irr{\operatorname{Irr}}
\def\Iso{\operatorname{Iso}}

\def\map{\operatorname{map}}

\def\proj{\ensuremath{\mathrm{proj}}}

\def\res{\operatorname{res}}

\def\susp{\operatorname{susp}}

\def\defeq{\overset{\text{def}}{=}}

\newcommand{\ul}[1]{\underline{#1}}
\newcommand{\xto}[1]{\xrightarrow{#1}}

\def\Ball{\operatorname{Ball}}
\def\half{\frac{1}{2}}
\def\inv{\operatorname{inv}}
\def\pinch{\operatorname{pinch}}
\def\red{\operatorname{red}}

\title[Self equivalences of equivariant spheres]{Stabilization of the homotopy groups of the self equivalences of equivariant spheres}

\author{Assaf Libman}
\address{Institute of Mathematics, University of Aberdeen, Fraser Noble Building, Aberdeen AB24 3UE, U.K.}
\email{a.libman@abdn.ac.uk}


\numberwithin{equation}{section}

\begin{document}

\begin{abstract}
Let $U_1,U_2,\dots$ be a sequence of orthogonal representations of a finite group $G$ such that every irreducible summand of $\oplus_{n} U_n$ has infinite multiplicity.
Let $V_n=\oplus_{i=1}^n U_n$ and $S(V_n)$ denote the sphere of unit vectors.
Then for any $i \geq 0$ the sequence of group $\dots \to \pi_i \map^G(S(V_n),S(V_n)) \to \pi_i \map^G(S(V_{n+1}),S(V_{n+1})) \to \dots$ stabilizes.
The stable group is a direct sum of $\omega_i(B N_GH/H)$ for a certain collection of subgroups $H$. % related to the isotropy group of $S(V_n)$.
% with the stable group $\oplus_H \omega_i(BW_GH)$ where $H$ runs through representatives of the conjugacy classes of all the isotropy groups of the points of $S(\oplus_n U_n)$.

%Let $G$ be a finite group.
%Let $U_1,U_2,\dots$ be a sequence of orthogonal representations in which any irreducible representation of $\oplus_{n \geq 1} U_n$ has infinite multiplicity.
%Let $V_n=\oplus_{i=1}^n U_n$ and $S(V_n)$ denote the linear sphere of unit vectors.
%Then for any $i \geq 0$ the sequence of group $\dots \to \pi_i \map^G(S(V_n),S(V_n)) \to \pi_i \map^G(S(V_{n+1}),S(V_{n+1})) \to \dots$ stabilizes with the stable group $\oplus_H \omega_i(BW_GH)$ where $H$ runs through representatives of the conjugacy classes of all the isotropy groups of the points of $S(\oplus_n U_n)$.
\end{abstract}

\maketitle

\section{Introduction}\label{S:Intro}
Let $G$ be a finite group.
A real representation $U$ of $G$ can be equipped with an essentially unique $G$-invariant norm.
The set $S(U)$ of unit vectors is called a {\em linear sphere} for $G$.
The one point compactification of $U$ is denoted $S^U$ with $\infty \in S^U$ as a basepoint.

This paper grew out of the interest in stabilization properties of the homotopy groups of the space $\map^G(S(U),S(U))$ of equivariant self maps.
To make this precise, let $U_1,U_2,\dots$ be a sequence of real representations of $G$.
Let $\Irr(U_\bullet)$ be the set of their irreducible summands.
Throughout we will assume: 
\begin{enumerate}[label={\bf(\Alph*)}]
\setcounter{enumi}{20}
\item
\label{H:U}
% If an irreducible representation of $G$ is a summand of some $U_n$ then it is a summand in infinitely many of them.
 Any $V \in \Irr(U_\bullet)$ has infinite multiplicity in $\oplus_{n\geq 1} U_n$.
%The multiplicity of any $V \in \Irr(U_\bullet)$ in $\oplus_{n\geq 1} U_n$ is infinite.
\end{enumerate}

A map of \emph{unpointed} spaces $f \colon X \to Y$ is called a {\em $k$-equivalence} %(Definition \ref{D:k-equivalence}) 
if it induces a bijection on path components, isomorphisms $\pi_i X \to \pi_i Y$ for all $1 \leq i \leq k$ and an epimorphism $\pi_{k+1}X \to \pi_{k+1}Y$ for any choice of basepoint in $X$.
Let $\omega_i(X)$ denote the stable homotopy groups of $X_+$ (the disjoint union of $X$ with a basepoint).
Let $BG$ denote the classifying space of a group $G$.
Let $(H)$ denote the conjugacy class of $H \leq G$ and $WH=N_G(H)/H$.
Let $\Iso_G(X)$ be the set of the isotropy groups of the points of a $G$-space $X$.
%If $G$ acts on a space $X$, set 
%\[
%\Iso_G(X)=\{ G_x : x \in X\} \qquad \text{(isotropy groups).}
%\]

If $U,V$ are orthogonal representations of $G$ then $S(U\oplus V)$ is homeomorphic to the join $S(U)*S(V)$.
There results $\map^G(S(U),S(U)) \xto{f \mapsto f*\id_{S(V)}} \map^G(S(U\oplus V),S(U\oplus V))$.

\begin{theorem}\label{T:linear spheres}
Let $U_1, U_2,\dots$ be a sequence of real representations of a finite group $G$ which satisfies hypothesis \ref{H:U}.
Let $U_{\leq n}$ denote $\oplus_{i=1}^n U_i$.
Then for any $k \geq 0$ the maps
\[
\map^G(S(U_{\leq n}),S(U_{\leq n})) \xto{f \mapsto f*\id_{S(U_{n+1})}} \map^G(S(U_{\leq n+1}),S(U_{\leq n+1}))
\]
are $k$-equivalences for all sufficiently large $n$.
In addition, provided $n$ is sufficiently large, there are isomorphisms (bijection if $i=0$) for any $0 \leq i \leq k$ and any choice of basepoint
\[
\pi_i \map^G(S(U_{\leq n}),S(U_{\leq n})) \cong \bigoplus_{(H) \subseteq \F(U_\bullet) } \omega_i(BWH)
\]
where $\F(U_\bullet)$ is the smallest collection of subgroups of $G$ which contains $\cup_{V \in \Irr(U_\bullet)} \Iso_G(S(V))$ and is closed under intersection of groups, and the sum is over its conjugacy classes.
%where the sum is over the conjugacy classes in $\F(U_\bullet)$ which is the smallest collection of subgroups of $G$ which contains $\cup_{V \in \Irr(U_\bullet)} \Iso_G(S(V))$ and is closed under intersection of groups.
\end{theorem}

Unreduced suspension gives a homeomorphism $\Sigma S(U) \cong S^U$.
Denote the fixed points $0,\infty \in S^U$ by $\alpha_0,\alpha_1$, the latter is the basepoint of $S^U$.
Let $\map^G(S^U,S^U;\id_{\{\alpha_0,\alpha_1\}})$ be the space of self maps which leave $\alpha_0$ and $\alpha_1$ fixed.

\begin{prop}\label{P:map S(U) to S^U}
With the set up and notation of Theorem \ref{T:linear spheres}, let $U=U_{\leq n}$ for sufficiently large $n$.
Then unreduced suspension gives a $k$-equivalence
\[
\map^G(S(U),S(U)) \xto{\susp \colon f \mapsto \Sigma f} \map^G(S^{U},S^{U};\id_{\{\alpha_0,\alpha_1\}}).
\]
\end{prop}

If $U,V$ are representations, there is a homeomorphism $S^U \wedge S^V \xto{u \wedge v \mapsto u+v} S^{U+V}$.
%There is a homeomorphism $S^{U+V} \xrightarrow{u+v \mapsto u \wedge v} S^U \wedge S^V$.

\begin{theorem}\label{T:tom Dieck extended}
With the set up and notation of Theorem \ref{T:linear spheres}, for every sufficiently large $n$, the map (between the {\em pointed} mapping spaces)
\[
\map^G_*(S^{U_{\leq n}},S^{U_{\leq n}}) \xto{f \mapsto f \wedge S^{U_{n+1}}} \map^G_*(S^{U_{\leq n+1}},S^{U_{\leq n+1}})
\]
is a $k$-equivalence.
If $Irr(U_\bullet)$ contains the trivial representation then for all $i \geq 0$ there are isomorphisms (bijection if $i=0$)
\[
\colim_n \pi_i \, \map^G_*(S^{U_{\leq n}},S^{U_{\leq n}}) \cong \underset{{(H) \subseteq \F(U_\bullet)}}{\oplus} \omega_i(BWH).
\]
If $Irr(U_\bullet)$ does not contain the trivial representation then 
\[
\colim_n \pi_i \, \map^G_*(S^{U_{\leq n}},S^{U_{\leq n}}) \cong \underset{{(H) \subseteq \F(U_\bullet)}}{\oplus} \omega_i(BWH \coprod BWH).
\]
\end{theorem}

\begin{void}
{\bf Relation with  tom Dieck splitting.}
Let $\mathbf{S}$ be the sphere spectrum in a category of $G$-spectra on a \emph{complete} universe \cite[Sec. I.2]{LMS86}.
tom Dieck's splitting \cite[Satz 2]{tomDieck75} 
%\footnote{Tammo tom Dieck, Satz 2 of Orbittypen und \"aquivariante Homologie. II., Arch. Math. (Basel) 26 (1975), no. 6, 650-662} 
yields $\pi_* (\mathbf{S}^G) \cong \oplus_H \, \omega_*(BWH)$ where $H$ runs through representatives of the conjugacy classes of \emph{all} the subgroups of $G$.
Theorem \ref{T:tom Dieck extended} can be viewed as an extension of this result to non-complete $G$-universes, and in fact, ``universes'' which do not contain the trivial representation.

Indeed, hypothesis \ref{H:U} makes $U_{\leq n}$ an indexing sequence in a $G$-universe $\U$ underlying a category of $G$-spectra in the sense of \cite[Sec. I.2]{LMS86}, with the only except that we do not insist that $\Irr(U_\bullet)$ contains the trivial representation.
If $\Irr(U_\bullet)$ contains all the irreducible representations of $G$ then $\U$ is a \emph{complete} $G$-universe and Theorem \ref{T:tom Dieck extended} recovers tom Dieck's splitting since $\Irr(U_\bullet)$ contains all $\Irr(\Ind_H^G(\RR))$, since $\mathbf{S} \simeq F(\mathbf{S},\mathbf{S})$, and since sequential colimits commute with $G$-fixed points and homotopy groups.
\end{void}

If the trivial representation is present in $\Irr(U_\bullet)$  then $S(U_{\leq n})$ has a fixed point and Theorem \ref{T:linear spheres} can be deduced from Hauschild's results \cite[Satz 2.4]{Hauschild77}.
%The difficulty in 
%the proof of 
%Theorem \ref{T:linear spheres} lies in the absence of fixed points in $S(U_n)$, i.e the absence of the trivial representation from $\Irr(U_\bullet)$.
%In its presence, Theorem \ref{T:linear spheres} can be deduced from Hauschild's results \cite[Satz 2.4]{Hauschild77}.
Things are harder in its absence.
Becker and Schultz proved the theorem in the case that $G$ acts freely \cite{BeSch73} using geometric methods.
Spectral sequence arguments were used by Schultz \cite[Prop. 6.5]{Schultz73II} to prove Theorem \ref{T:linear spheres} when $G$ is cyclic.
Klaus proved that for any $k \geq 1$ the groups  $\pi_{k} \map^G(S(U_{\leq n}),S(U_{\leq n}))$, where $\id$ is the basepoint, are finite  for all sufficiently large $n$ \cite[Proposition 2.5]{Klaus11}.
The author improved this result in \cite{Lib16}, giving bounds for their order (uniform in $n$ for each $k$).
However, to establish Theorem \ref{T:linear spheres} in its full generality spectral sequences become unmanageable. % and new ideas are needed.

%Our proof of Theorem \ref{T:linear spheres} makes complete separation between the stabilization and the identification of the stable groups.
The stabilization statement in Theorem \ref{T:linear spheres} is a consequence of Proposition \ref{P:stabilisation main result} while the identification of the limit groups is a special case of Proposition \ref{P:stable term X}.
These propositions are the main technical results of this paper.
They allow, in principle, to prove Theorem \ref{T:linear spheres} for a larger class of spaces than linear spheres.
Their hypotheses, though, are very restrictive. % so restrictive that it is hard to find $G$-spaces other than linear spheres to which they can apply.

{\bf Acknowledgements:} 
I would like to thank Michael Crabb who read early versions of this paper and shared his ideas and most importantly, suggesting the description of the limit groups in Theorem \ref{T:linear spheres} as a direct sum of stable homotopy groups.
I would also like to thank Irakli Patchkoria for helpful discussions. % about Theorem \ref{T:tom Dieck extended}.

%----------------------
%----------------------
\section{Preliminaries}
\label{Sec:Preliminaries}

\begin{defn}\label{D:k-equivalence}
A map $f \colon X \to Y$ of spaces is called a $k$-equivalence if it is bijective on components and for any $x \in X$ the induced maps $\pi_i(X,x) \to \pi_i(Y,f(x))$ are isomorphisms for all $1 \leq i \leq k$ and epimorphism for $k+1$.
\end{defn}

This is just a (convenient) ``shift by $1$'' of the standard definitions of $n$-connectedness of maps, see \cite{Wh78}.
A space $X$ is called $k$-connected, where $k \geq 0$, if $X \to *$ is a $k$-equivalence.
We will write
\[
\conn X =k.
\]
By convention $\conn X=-1$ if the number of path components of $X$ is not $1$.
The next two results are straightforward.

%
% -
%
\begin{lemma}\label{L:k equivalences two out of three}
Let $k \geq 0$.
Consider a morphism of fibre sequences where $b_1 \in B_1$.
\[
\xymatrix{
\Fib(p_1,b_1) \ar[r] \ar[d]^{h} & E_1 \ar@{->>}[r]^{p_1} \ar[d]^{g} & B_1 \ar[d]^{f} \\
\Fib(p_2,f(b_1)) \ar[r] & E_2 \ar@{->>}[r]_{p_2}  & B_2.
}
\]
\begin{enumerate}
\item
\label{L:k equivalences 2 outof 3:total}
If $f$ is a $k$-equivalence and the map $h$ is a $k$-equivalence for any choice of $b_1 \in B_1$ then $g$ is a $k$-equivalence.
\item
\label{L:k equivalences 2 outof 3:fibres}
If $f$ is a $(k+1)$-equivalence and $g$ is a $k$-equivalence then $h$ is a $k$-equivalence for any choice of $b_1 \in B_1$.
\end{enumerate}
\end{lemma}
\begin{proof}
This is standard diagram chase of exact sequences of pointed sets and groups.
The first assertion is slightly more delicate in connection to the surjection on components and uses the homotopy lifting property of $p_2$.
\end{proof}

\begin{lemma}\label{L:k equivalence transitive fibres}
Let $k \geq 0$ and consider the following ladder of Serre fibrations
\[
\xymatrix{
E_1 \ar[d]_h \ar@{->>}[r]^{q_1} & D_1 \ar[d]^{g} \ar@{->>}[r]^{p_1} & B_1 \ar[d]^f 
\\
E_2 \ar@{->>}[r]_{q_2} & D_2 \ar@{->>}[r]_{p_2} & B_2 
}
\]
Assume that $f,g,h$ induce $k$-equivalences on the fibres of $p_1$ and $p_2$ and on the fibres of $q_1$ and $q_2$.
Then they induce $k$-equivalences on the fibres of $p_1 \circ q_1$ and $p_2 \circ q_2$.
\end{lemma}

\begin{proof}
Choose $b_1 \in B_1$ and set $b_2=f(b_2)$.
We need to show that $F_i=\Fib(p_i \circ q_i, b_i)$ are $k$-equivalent.
Set $X_i=\Fib(p_i,b_i)$.
We obtain a morphism of fibrations
\[
\xymatrix{
F_1 \ar@{->>}[r]^{q_1|_{F_1}} \ar[d]_{h|_{F_1}} & X_1 \ar[d]^{g|_{X_1}} 
\\
F_2 \ar@{->>}[r]_{q_2|_{F_2}} & X_2 
}
\]
By hypothesis $g|_{X_1}$ is a $k$-equivalence, so by Lemma \ref{L:k equivalences two out of three}\ref{L:k equivalences 2 outof 3:total} it remain to show that all the fibres of the rows of this diagram are $k$-equivalent.
These fibres are equal to the fibres of $q_1$ and $q_2$ and by the hypothesis they are $k$-equivalent.
\end{proof}

Throughout this paper we will work in a ``convenient category of $G$-spaces'', that is the category $\cgwh$ of compactly generated weak Hausdorff spaces, or the category $\cgh$ of compactly generated Hausdorff spaces, see \cite{St66} or \cite{Stlnd09}.
This category has products and function complexes $F(X,Y)$ giving adjunction homeomorphisms $\map(Z \times X, Y) \cong \map(Z,F(X,Y))$ where $\map$ denotes the set of morphisms in $\cgwh$.
In fact, $\cgwh$ is enriched over itself and  $\map(X,Y) \cong F(X,Y)$.

Let $G$ be a discrete group, e.g finite.
Let $G-\cgwh$ be the category of $G$-spaces. % spaces equipped with action of $G$ with $G$-maps between them.
Regarding $X$ and $Y$ as objects in $\cgwh$ via the forgetful functor, $F(X,Y)$ is equipped with a standard action of $G$ where $(g \cdot \vp)(x) = g\vp (g^{-1}x)$.
In this way the set $\map^G(X,Y)$ of all $G$-maps $X \to Y$ is equipped with a topology giving rise to the adjunction homeomorphism $\map^G(Z \times X, Y) \cong \map^G(Z,F(X,Y))$.

Let $Y \subseteq X$ be an inclusion of $G$-spaces.
We denote by $G_x$ the stabilizer of $x \in X$.
Set
\[
\Iso_G(X,Y)=\{ G_x : x \in X \setminus Y\}.
\]
If $Y= \emptyset$ we will simply write $\Iso_G(X)$.

An inclusion $B \subseteq A$ of $G$-spaces is a {\em relative $G$-CW complex} if $A$ is obtained from $B$ by attaching equivariant cells.
A $G$-CW complex is a space obtained in this way from the empty set.
We emphasize that by $G$-CW complexes we always mean that $G$ acts cellularly (by permuting cells).
See \cite[Chapter II.1]{tomDieck87}.
For $H \in \Iso_G(A,B)$, the {\em $H$-relative dimension} is 
\[
\dim_H(A,B) \defeq \dim \, (A^H,B^H).
\]
This is the maximum dimension of an equivariant cell of type $G/H$ in $A$ which is not contained in $B$.
For any $G$-space $Y$, a relative $G$-CW complex $(A,B)$ gives rise to a Serre fibration  $\map^G(A,Y) \to \map^G(B,Y)$.

A map of $G$-spaces is a $k$-equivalence if this is the case by forgetting the action of $G$.

%
% -
%
\begin{lemma}\label{L:k equivalence on fibres}
Fix some $k \geq 0$.
Suppose $B \subseteq A$ is an inclusion of finite dimensional $G$-CW complexes and $f \colon X \to Y$ is a map of $G$-spaces.
For any $H \in \Iso_G(A,B)$ set $n_H=\dim_H(A,B)$ and assume that the map $X^H \xto{f^H} Y^H$ is a $(k+n_H)$-equivalence.
Then for any $\vp \in \map^G(B,X)$ the map $f_*$ induced on fibres in 
\[
\xymatrix{
\Fib(i^*,\vp) \ar[r] \ar[d]_{f_*} & \map^G(A,X) \ar@{->>}[r]^{i^*} \ar[d]_{f_*} & \map^G(B,X) \ar[d]^{f_*} \\
\Fib(i^*,f \circ \vp) \ar[r] &  \map^G(A,Y) \ar@{->>}[r]^{i^*} & \map^G(B,Y) 
}
\]
is a $k$-equivalence of spaces.
\end{lemma}

\begin{proof}
We use induction on $n=\dim(A,B)$.
If $n=-1$ then $A=B$ and the result is trivial.
Assume that $n \geq 0$ and let $A' \subseteq A$ be the union of the $(n-1)^\text{st}$ skeleton of $A$ with $B$. %set $A'=\OP{sk}_{n-1}(A) \cup B$.
Let $i \colon B \to A$ and $j \colon B \to A'$ and $\ell \colon A' \to A$ denote the inclusions.
We obtain a diagram of fibrations
\begin{equation}\label{E:k equivalences step 1}
\vcenter{
\xymatrix{
\map^G(A,X) \ar[d]^{f_*} \ar@{->>}[r]^{\ell^*} &
\map^G(A',X) \ar[d]^{f_*} \ar@{->>}[r]^{j^*} &
\map^G(B,X) \ar[d]^{f_*} \\
\map^G(A,Y) \ar@{->>}[r]^{\ell^*} &
\map^G(A',Y) \ar@{->>}[r]^{j^*} &
\map^G(B,Y) 
}
}
\end{equation}
such that composition of the rows are the maps $i^*$. 
By construction $\dim (A',B) \leq n-1$ and also $\dim_H(A',B) \leq \dim_H(A,B)=n_H$ for any $H \in \Iso_G(A',B)$.
The induction hypothesis applies to the inclusion $B \subseteq A'$ and we deduce that the fibres of $j^*$ in the second square are $k$-equivalent.
% vertical arrow on the right of \eqref{E:k equivalences step 2} is a $k$-equivalence.
By Lemma \ref{L:k equivalence transitive fibres} it remains to show that the fibres of the maps $\ell^*$ are $k$-equivalent.

Since $A$ is obtained from $A'$ by attaching equivariant $n$-cells, we get a pushout diagram
\[
\xymatrix{
\coprod_{i \in \A} S^{n-1} \times G/H_i \ar@{^(->}[d] \ar[r]^(0.7){\eta} & A' \ar@{^(->}[d] \\
\coprod_{i \in \A} D^{n} \times G/H_i \ar[r] & A.
}
\]
where $\A$ indexes the equivariant $n$-cells attached to $A'$, hence $n_{H_i}=n$ for all $i \in \A$.
By applying $\map^G(-,T)$, where $T$ is any $G$-space, we obtain a pullback diagram, natural in $T$,
\[
\xymatrix{
\map^G(A,T) \ar@{->>}[r] \ar[d] &
\map^G(A',T) \ar[d]^{\eta^*} 
\\
\prod_i \map(D^n,T^{H_i}) \ar@{->>}[r] & \prod_i \map(S^{n-1},T^{H_i}).
}
\]
Since this is is a pullback square, the fibres of the horizontal maps are homeomorphic.
Applying this for $T=X$ and $T=Y$, it remains to prove that the maps on all fibres in the following commutative diagram
\[
\xymatrix{
\prod_i \map(D^n,X^{H_i}) \ar@{->>}[r] \ar[d]^{f_*} & \prod_i \map(S^{n-1},X^{H_i}) \ar[d]^{f_*} \\
\prod_i \map(D^n,Y^{H_i}) \ar@{->>}[r] & \prod_i \map(S^{n-1},Y^{H_i})
}
\]
are $k$-equivalences.
% By construction $n=\dim_{H_i}(A,A') \leq \dim_{H_i}(A,B)=n_{H_i}$, hence $n_{H_i}=n$ for all $i$.
If $n=0$ then the spaces on the right are points and the map of the spaces on the left is by hypothesis $n_{H_i}+k=n+k=k$ equivalence, so the map on fibres is a $k$-equivalence.
If $n \geq 1$ then by the hypothesis the vertical arrow on the left is a $(k+n)$-equivalence, hence a $k$-equivalence, and the vertical arrow on the right is a $k+n-(n-1)=k+1$ equivalence.
Lemma \ref{L:k equivalences two out of three} shows that the map on all fibres is a $k$-equivalence and this completes the proof.
\end{proof}

\begin{cor}\label{C:maps with highly connected target Y}
Let $k \geq 0$.
Let $A$ be a finite $G$-CW complex and $B$ be a subcomplex.
Let $Y$ be a $G$ space  and let $k \geq 0$.
Assume that $\conn Y^H - \dim_H(A,B) \geq k$ for every $H \in \Iso_G(A,B)$.
Then $\map^G(A,Y) \xto{i^*} \map^G(B,Y)$ is a $k$-equivalence.
\end{cor}
\begin{proof}
Lemma \ref{L:k equivalence on fibres} with $B \subseteq A$ and $Y \to *$ shows that all the fibres of $\map^G(A,Y) \to \map^G(B,Y)$ are $k$-connected.
Then apply Lemma \ref{L:k equivalences two out of three}\ref{L:k equivalences 2 outof 3:total} with $g=i^*$ and $f$ the identity on $\map^G(B,Y)$.
\end{proof}

\begin{void}\label{V:eta}
Let $\Sigma X$ be the \emph{unreduced} suspension of a space $X$.
The images of $0,1 \in I$ yield two distinguished points
\[
\alpha_0,\alpha_1 \in \Sigma X,
\]
of which $\alpha_1$ is chosen as the basepoint.
If $X$ is a $G$-space both $\alpha_0,\alpha_1$ are fixed points.

Let $Y$ be a space and $y_0,y_1 \in Y$ points.
Let $F(\Sigma X,Y;y_0,y_1)$ be the subspace of maps $f \colon \Sigma X \to Y$ such that $f(\alpha_0)=y_0$ and $f(\alpha_1)=y_1$.
Let $\PP_{y_0,y_1}Y$ denote the space of paths $\omega \colon I \to Y$ such that $\omega(0)=y_0$ and $\omega(1)=y_1$.
There is an adjunction homeomorphism
\[
F(\Sigma X,Y;y_0,y_1) \cong F(X,\PP_{y_0,y_1}Y).
\]
In particular $F(\Sigma X,\Sigma X;\alpha_0,\alpha_1) \cong F(X,\PP_{\alpha_0,\alpha_1} \Sigma X)$ and denote the adjoint of the identity by
\[
\eta_X \colon X \to \PP_{\alpha_0,\alpha_1} \Sigma X \qquad (\eta_X(x)(t) = (t,x) \in \Sigma X)
\]
If $X$ is a $G$-space then $\eta_X$ is a $G$ map, $(\eta_X)^H = \eta_{X^H}$ for any $H \leq G$, and $\eta_X$ the unit of the adjunction isomorphism
\[
\map^G(\Sigma X, \Sigma X) \cong \map^G(X, \PP_{\alpha_0,\alpha_1} \Sigma X).
\]
\end{void}

%Freudenthal's suspension theorem is the claim that $F_*(S^n,S^m) \xto{\susp_*} F_*(S^{n+1},S^{m+1})$ induced by the reduced suspension functor is a $(2m-2)$-equivalence ($F_*$ denotes spaces of pointed maps).
%For the proof of Theorem \ref{T:linear spheres} 

We will need the following (somewhat expected) corollary of Freudenthal's theorem.

\begin{prop}
\label{P:Freudenthal consequence new}
\begin{enumerate}[label=(\alph*)]
\item
\label{P:Freudenthal a}
The map $\eta_{S^n} \colon S^n \to \PP_{\alpha_0,\alpha_1}\Sigma S^n$ is a $(2n-2)$-equivalence, where $n \geq 1$.

\item
\label{P:Freudenthal b}
Suppose $m>n\geq 1$.
Then $F(S^n,S^m) \xto{\susp \colon f \mapsto \Sigma f} F(\Sigma S^n,\Sigma S^m)$ is a $(m-1)$-equivalence of path connected spaces.
\end{enumerate}
\end{prop}

\begin{proof}
\ref{P:Freudenthal a}
Choose a basepoint $x_0 \in S^n$ and consider the  quotient map $\pi \colon \Sigma S^n \to \Sigma^{\red}S^n$.
%By collapsing $I \times \{x_0\} \subseteq \Sigma S^n$ to a point we obtain the quotient map $\pi \colon \Sigma S^n \to \Sigma^{\red}S^n$.
By inspection $\pi \circ \eta_{S^n}$ is the canonical map $S^n \to \Omega \Sigma^{\red} S^n$ which is a $(2n-2)$-equivalence by Freudenthal's theorem.
The result follows since $\pi$ is a homotopy equivalence.

\ref{P:Freudenthal b}
By definition of $\eta_{S^n}$ the triangle in the following diagram commutes
\[
\xymatrix{
F(S^n,S^m) \ar[r]^{(\eta_{S^n})_*} \ar[d]_{f \mapsto \Sigma f} &
F(S^n,\PP_{\alpha_0,\alpha_1} \Sigma S^m) \ar[dl]^{\cong}  
\\
F(\Sigma S^n,\Sigma S^m; \alpha_0,\alpha_1) \ar@{^(->}[r] &
F(\Sigma S^n,\Sigma S^m) \ar@{->>}[r]^{(\ev_{\alpha_0},\ev_{\alpha_1})} &
\Sigma S^m \times \Sigma S^m.
}
\]
The inclusion of the fibre in the second row is a $(m-1)$-equivalence.
The vertical arrow is a $(2m-2-n)$-equivalence by part \ref{P:Freudenthal a} and Lemma \ref{L:k equivalence on fibres} applied to $\emptyset \subseteq S^n$ and $\eta_{S^m}$, and $F(S^n,S^m)$ is clearly path connected.
Since $m-1 \leq 2m-2-n$, the result follows.
\end{proof}

%%%%
%%%%
%%%%
\section{Square diagrams of spaces}
\label{Sec:squares}

Let $\PP X = F(I,X)$ denote the path space of $X$.
% For a space $X$ let $\PP X$ denote the path space $\map(I,X)$.
The homotopy pullback of a diagram of spaces $X_0 \xto{f} X_2 \xleftarrow{g} X_1$ is the subspace of $X_0 \times X_1 \times \PP X_2$ consisting of $(x_0,x_1,\omega)$ such that $f(x_0)=\omega(0)$ and $g(x_1)=\omega(1)$.
The homotopy fibre of $X \xto{f} Y$ over $y_0 \in Y$ is the homotopy pullback of $X \xto{f} Y \xleftarrow{y_0} *$.
There is an inclusion $\Fib(f,y_0) \subseteq \hofib(f,y_0)$ via the constant paths, and if $f$ is a Serre fibration this inclusion is a weak homotopy equivalence.

\begin{defn}\label{D:cat sq}
Let $\catsq$ be the category whose objects are commutative diagrams of spaces % of the form
\begin{equation}\label{E:object of catsq}
\A=
\qquad
\vcenter{
\xymatrix{
A_3 \ar[r]^{a_{32}} \ar[d]_{a_{31}} & 
A_2 \ar[d]^{a_{20}} 
\\
A_1 \ar[r]_{a_{10}} & 
A_0
}
}
\end{equation}
Morphisms are natural transformations of diagrams.
Thus, a morphism $\vp \colon \A \to \B$ is a quadruple $(\vp_0,\vp_1,\vp_2,\vp_3)$ of maps $\vp_i \colon A_i \to B_i$ with the obvious commutation relations with the structure maps of $\A$ and $\B$.

A {\em basepoint} for $\A$ is a triple of $\ul{x}=(x_0,x_1,x_2) \in A_0 \times A_1 \times A_2$ such that $a_{20}(x_2)=x_0$ and $a_{10}(x_1)=x_0$. 
Notice that we do not choose $x_3 \in A_3$ compatible with $x_0,x_1,x_2$.
\end{defn}
\noindent
A basepoint of $\A \in \catsq$ gives rise to the following diagram of spaces which we denote by $(\A,\ul{x})$.
\[
(\A,\ul{x}) \qquad = \qquad
\vcenter{
\xymatrix{
A_3 \ar[r]^{a_{32}} \ar[d]_{a_{31}} & 
A_2 \ar[d]^{a_{20}} & 
\bullet \ar[l]_{x_2} \ar[d]
\\
A_1 \ar[r]_{a_{10}} & 
A_0 
& \bullet \ar[l]_{x_0} 
\\
\bullet \ar[r] \ar[u]^{x_1} & 
\bullet \ar[u] & 
\bullet \ar[l] \ar[u]
}
}
\]
We obtain a category $\catsq_*$ whose objects are $(\A,\ul{x})$ with natural transformations between them.
The homotopy limit functor gives rise to a functor $\Lambda \colon \catsq_* \to \spaces$
\begin{equation}\label{E:def Lambda}
\Lambda(\A,\ul{x}) \, \stackrel{\text{def}}{=} \, \holim (\A,\ul{x}).
\end{equation}
Fubini's theorem for homotopy limits \cite[Secs. 24 and 31]{Ch-Sch02} implies that $(\A,\ul{x})$  can be calculated by first taking the homotopy limits of the rows (resp. columns) and then take the homotopy limits of the resulting pullback diagram of spaces.
% Since the homotopy limits of the rows (resp. columns) of $(\A,\ul{x})$ are the homotopy fibres of the rows of $\A$ over $x_2$ and $x_0$ (resp. the homotopy fibres of the columns of $\A$ over $x_1$ and $x_0$), it follows that
Therefore
\begin{eqnarray}\label{E:Fubini for Lambda}
&& \Lambda(\A,\ul{x}) \, \cong \, \hofib\left(\, \hofib(a_{32},x_2) \xto{\ (a_{31},a_{20}) \ } \hofib(a_{10},x_0)\, ,\, x_1 \right) \\
\nonumber
&& \Lambda(\A,\ul{x}) \, \cong \, \hofib\left(\, \hofib(a_{31},x_1) \xto{\ (a_{32},a_{10}) \ } \hofib(a_{20},x_0)\, ,\, x_2 \right)
\end{eqnarray}
where $x_1 \in \Fib(a_{10},x_0) \subseteq \hofib(a_{10},x_0)$ and $x_2 \in \Fib(a_{20},x_0) \subseteq \hofib(a_{20},x_0)$.

\begin{lemma}\label{L:elementary facts on Lambda}
Let $\vp \colon \A \to \B$ be a morphism in $\catsq$ depicted by the vertical arrows in
\[
\xymatrix@!0{
A_3 \ar[rr]^{a_{32}} \ar[dd]_{\vp_3} \ar[dr] & &                       % back top left
A_2 \ar[dr]^{a_{20}} \ar'[d][dd]                               % back top right
\\
& A_1 \ar[rr] \ar[dd] & &                             % front top left
A_0 \ar[dd]^{\vp_0}                                           % front top right
\\
B_3 \ar[dr]_{b_{31}} \ar'[r][rr]  & &                          % back bottom left
B_2 \ar[dr]                                           % back bottom right
\\
& B_1 \ar[rr]_{b_{10}} & &                                     % front bottom left
B_0
}
\iffalse
\qquad \text{Or ??}
\xymatrix@!0{
A_3 \ar[rr]^(0.3){\vp_3} \ar[dd] \ar[dr] & &                       % back top left
B_3 \ar[dr] \ar'[d][dd]                               % back top right
\\
& A_2 \ar[rr]^(0.3){\vp_2} \ar[dd] & &                             % front top left
B_2 \ar[dd]                                           % front top right
\\
A_1 \ar[dr] \ar'[r]^(0.7){\vp_1}[rr]  & &                          % back bottom left
B_1 \ar[dr]                                           % back bottom right
\\
& A_0 \ar[rr]^(0.3){\vp_0} & &                                     % front bottom left
B_0
}
\fi
\]

\begin{enumerate}
\item
\label{L:elementary facts on Lambda 1}
If the side faces (or the back and front faces) are homotopy pullback squares then the induced map $\Lambda(\A,\ul{x}) \xto{\Lambda(\vp)} \Lambda(\B,\vp(\ul{x}))$ is a (weak) homotopy equivalence for any choice of basepoint $\ul{x}$ for $\A$.

\iffalse
\item
\label{L:elementary facts on Lambda 2}
{\bf [NOT USED ?]}
If $\vp_0$ is a $(k+2)$-equivalence, $\vp_1$ and $\vp_2$ are $(k+1)$-equivalences and $\vp_3$ is a $k$-equivalence then $\Lambda(\A,\ul(a)) \to \Lambda(\B,\ul{b})$ is a $k$-equivalence.
%If the vertical arrows at the left face (those of $\vp$) are $(k+1)$-equivalences and those on the right face are $(k+2)$-equivalences, then $\Lambda(\A,\ul(a)) \to \Lambda(\B,\ul{b})$ is a $k$-equivalence.
\fi

\item
\label{L:elementary facts on Lambda 3}
Let $\ul{x}$ be a basepoint for $\A$. % and let $\ul{y}=\vp(\ul{x})$ be a basepoint in $\B$.
Suppose that 
\begin{itemize}
\item[(i)] $\vp_2$ and $\vp_0$ induce a $(k+1)$-equivalence $\hofib(a_{20},x_0) \to \hofib(b_{20},\vp_0(x_0))$ and 
\item[(ii)]  $\vp_1, \vp_3$ induce a $k$-equivalence $\hofib(a_{31},x_1) \to \hofib(b_{31},\vp_1(x_1))$.
\end{itemize}
Then $\Lambda(\A,\ul{x}) \xto{\Lambda(\vp)} \Lambda(\B,\vp(\ul{x}))$ is a $k$-equivalence.

\item 
\label{L:elementary facts on Lambda 4}
Suppose that $A_2 \xto{a_{20}} A_0$ is a Serre fibration and that
\begin{itemize}
%\item[(i)] $(A_2 \sslash A_0,x_0) \to (B_2\sslash B_0,\vp(x_0))$ is a $k$-equivalence for any basepoint $x_0 \in A_0$, and
\item[(i)] $\hofib(a_{20},x_0) \to \hofib(b_{20},\vp_1(x_0))$ is a $k$-equivalence for any basepoint $x_0 \in A_0$, and
\item[(ii)] $\Lambda(\A,\ul{x}) \xto{\Lambda(\vp)} \Lambda(\B,\vp(\ul{x}))$ is a $k$-equivalence for any choice of basepoint $\ul{x}$ in $\A$.
\end{itemize}
%Then $(A_3\sslash A_1,x_1) \to (B_3\sslash B_1,\vp(x_1))$ is a $k$-equivalence for any basepoint $x_1 \in A_1$.
Then $\hofib(a_{31},x_1) \to \hofib(b_{31},\vp(x_1))$ is a $k$-equivalence for any basepoint $x_1 \in A_1$.
\end{enumerate}
\end{lemma}
\begin{proof}
\ref{L:elementary facts on Lambda 1}
Since each side face is a homotopy pullback square, the induced maps on homotopy fibres of its rows is a  weak equivalence.
Thus, the vertical arrows in 
\begin{equation}\label{E:homotopy fibres d1}
\xymatrix{
\hofib(a_{31},x_1) \ar[r]^{(a_{32}, a_{10})} \ar[d]_{(\vp_3, \vp_1)} &
\hofib(a_{20},x_0) \ar[d]^{(\vp_2, \vp_0)} 
\\
\hofib(b_{31},\vp_1(x_1)) \ar[r]_{(b_{32}, b_{10})}  &
\hofib(b_{20},\vp_0(x_0))
}
\end{equation}
are weak homotopy equivalences.
Therefore the map induced on the homotopy fibres of the rows over $x_2$ and $\vp_2(x_2)$ are weak equivalences, and the result follows from \eqref{E:Fubini for Lambda}.

\noindent
\ref{L:elementary facts on Lambda 3}
This is immediate from \eqref{E:Fubini for Lambda} and Lemma \ref{L:k equivalences two out of three}\ref{L:k equivalences 2 outof 3:fibres}.

\noindent
\ref{L:elementary facts on Lambda 4}
Choose $x_1 \in A_1$ and set $x_0=a_{10}(x_1)$.
There results a commutative diagram as in \eqref{E:homotopy fibres d1}.
\iffalse
\[
\xymatrix{
\hofib(a_{31},x_1) \ar[r]^{(\PP a_{32},a_{10})} \ar[d]_{(\PP \vp_3,\vp_1)} &
\hofib(a_{20},x_0) \ar[d]^{(\PP\vp_2,\vp_0)} 
\\
\hofib(b_{31},\vp(x_1) \ar[r]^{(\PP b_{32},b_{10})} &
\hofib(b_{20},\vp(x_0)).
}
\]
\fi
By the hypothesis the vertical arrow on the right of \eqref{E:homotopy fibres d1} is a $k$-equivalence, and our goal is to show the the same is true for the vertical arrow on the left.
By Lemma \ref{L:k equivalences two out of three}\ref{L:k equivalences 2 outof 3:total} it remain to show that for any $x_2 \in \hofib(a_{20},x_0)$ the map induced on the homotopy fibres of the horizontal arrows is a $k$-equivalence.
By hypothesis $A_2 \to A_0$ is a Serre fibration, so the inclusion $\fib(a_{20},x_0) \subseteq \hofib(a_{20},x_0)$ is a weak homotopy equivalence.
Therefore we may consider only $x_2 \in \fib(a_{20},x_0)$ in which case $\ul{x}=(x_0,x_1,x_2)$ forms a basepoint for $\A$ and it follows from \eqref{E:Fubini for Lambda} that the map of the homotopy fibres over $x_2$ and $\vp_2(x_2)$ in the diagram \eqref{E:homotopy fibres d1} is the map $\Lambda(\A,\ul{x}) \xto{\Lambda(\vp)} \Lambda(\B,\vp(\ul{x}))$ which by hypothesis is a $k$-equivalence.
This completes the proof.
\end{proof}

Given an object $\A \in \catsq$ let $I \times \A$ be the object in $\catsq$ obtained by applying the functor $I \times - $ objectwise.
Similarly $\PP \A$ is %the object in $\catsq$ 
obtained by applying the path space functor $\PP(-)$ objectwise.

\begin{defn}\label{D:homotopy in catsq}
Let $\A,\B$ be objects in $\catsq$.
A {\em homotopy} is a morphism $\vp \colon I \times \A  \to \B$.
We frequently refer to a homotopy as a family of morphisms $\vp_p \colon \A \to \B$ (parameterized by $0 \leq p \leq 1$).
\end{defn}

The adjoint of a homotopy $\vp \colon I \times \A \to \B$ is a morphism $\vp^\# \colon \A \to \PP \B$.
If $\ul{x}$ is a basepoint of $\A$ then $\vp^\#(\ul{x})$ is a basepoint in $\PP \B$.
Evaluation at $p \in I$ gives a morphism $(\PP\B,\vp^\#(\ul{x})) \xto{\ev_p} (\B, \vp_p(\ul{x}))$ in $\catsq_*$ which is an object-wise homotopy equivalence.
We obtain a weak homotopy equivalence $\Lambda(\PP\B,\vp^\#(\ul{x})) \xto{\Lambda(\ev_p)} \Lambda(\B,\vp_p(\ul{x}))$.
The following lemma is an immediate consequence.
% Lemma \ref{L:elementary facts on Lambda} immediately implies

\begin{lemma}\label{L:homotopic morphisms in catsq}
Let $\vp_p \colon \A \to B$ be a homotopy in $\catsq$.
Then for any basepoint $\ul{x}$ in $\A$ there is a commutative diagram in which both evaluation morphisms are (weak) homotopy equivalences
\[
\xymatrix{
\Lambda(\A,\ul{x}) \ar[d]_{\Lambda(\vp_{p=1})} \ar[r]^{\Lambda(\vp_{p=0})} \ar[dr]^{\Lambda(\vp^\#)} &
\Lambda(\B,\vp_{p=0}(\ul{x}))
\\
\Lambda(\B,\vp_{p=1}(\ul{x})) &
\Lambda(\PP\B,\vp^\#(\ul{x})) \ar[u]^{\simeq}_{\Lambda(\ev_0)} \ar[l]_{\simeq}^{\Lambda(\ev_1)}
}
\]
In particular, if $\alpha, \beta \colon \A \to \B$ are homotopic morphisms and  $\ul{x}$ of $\A$ is a basepoint then $\Lambda(\alpha, \ul{x})$ is a $k$-equivalence if and only if $\Lambda(\beta,\ul{x})$ is.
 %$\Lambda(\vp_{p=0})$ is a $k$-equivalence if and only if $\Lambda(\vp_{p=1})$ is a $k$-equivalence.
\end{lemma}

%
% ---
%
\section{Join of spaces}\label{Sec:join}

The join %of spaces $X_1, \dots, X_n$ denoted 
$X_1* \dots * X_n$ is the homotopy colimit of the diagram of spaces indexed by the opposite category of the poset of the non-empty subsets $\sigma$ of $[n]=\{1,\dots,n\}$, and consisting of the spaces $X_\sigma \defeq \prod_{i \in \sigma} X_i$ and projection maps between them.
Let $\Delta^{n-1}=\{ (t_1,\dots,t_n) : t_i \geq 0, \sum_i t_i=1\}$ be the standard $(n-1)$-simplex in $\RR^n$.
The underlying set of the join is the set of equivalence classes of 
\begin{equation}\label{E:join quotient}
\coprod_{\emptyset \neq \sigma \subseteq [n]} \Delta^{|\sigma|-1} \times X_\sigma
\end{equation}
where for any $\tau \subseteq \sigma$ we declare $(s_i,x_i)_{i \in \tau} \sim (t_i,y_i)_{i \in \sigma}$ if $x_i=y_i$ and $s_i=t_i$ for all $i \in \tau$ and $t_i=0$ for all $i \in \sigma \setminus \tau$.
There are two natural choices to topologize the join, but when $X_1,\dots,X_n$ are compact Hausdorff both agree with the quotient topology, see \cite[Section 2]{Mil56II}. 
% (and notice that when the spaces $X_i$ are compact Hausdorff the weak and strong topologies on the join coincide).
% In light of all this, we will use the following notation throughout.

\begin{void}{\bf NOTATION:} Since the join will play a key role in this paper we will write
\[
X_1 \cdots X_n \qquad \text{instead of} \qquad X_1* \dots * X_n.
\]
%for the join $X_1* \dots * X_n$.
Its points are equivalence classes $[t_1x_1,\dots,t_nx_n]$ where $(t_1,\dots,t_n) \in \Delta^{n-1}$ and it is understood that $t_ix_i$ may be omitted from the notation if either $X_i$ is empty or if $t_i=0$, and two such brackets represent the same point if they agree except in the entries where $t_i=0$.
\end{void}

Identify $\Delta^1$ with the unit interval $I$ via $I \xto{t \mapsto (t,1-t)} \Delta^1$.
Then the join $XY$ of compact Hausdorff spaces 
%is the quotient space of $(X \times Y \times I) \coprod X \coprod Y$ under the equivalence relation generated by $(x,y,0) \sim y$ and $(x,y,1) \sim x$. %, namely $XY = \hocolim \ ( \xymatrix@1{X & X \times Y \ar[l]_{\pi_X} \ar[r]^{\pi_Y} & Y} )$.
%Thus, there is a pushout square
fits in a pushout diagram
\begin{equation}\label{E:join pushout}
\vcenter{
\xymatrix{
(X \times Y) \coprod (X \times Y) \ar@{^(->}[r]^(0.6){i_0 + i_1} \ar[d]_{\pi_Y \coprod \pi_X} &
X \times Y \times I \ar[d]^{\pi} 
\\
Y \coprod X \ar@{^(->}[r] & X*Y
}
}
\end{equation}

%The join construction is functorial.
If $X_1,\dots,X_n$ are $G$-spaces then their join is also a $G$-space via the diagonal action.
Given a $G$-space $Z$ we obtain a functor $X \mapsto XZ$ from $G-\cgwh$ to itself.
This functor is, in fact, continuous in the sense that for any $G$-spaces $X,Y$ the resulting natural map
\begin{equation}\label{E:def JZ}
%J_Z \colon \map^G(X,Y) \xto{f \mapsto f*\id_Z} \map^G(XZ,YZ)
J_Z \colon F(X,Y) \xto{f \mapsto f*\id_Z} F(XZ,YZ)
\end{equation}
is a continuous map.
% This map  will play a central role in this paper.
We can describe it explicitly: for any $f \in F(X,Y)$
\begin{equation}\label{E:JZ explicit formula}
J_Z(f)([sx,(1-s)z])=[s \cdot f(x),(1-s)z].
\end{equation}
One easily checks that $J_Z$ is  $G$-equivariant and passage to fixed points gives
\[
J_Z \colon \map^G(X,Y) \xto{f \mapsto f*\id_Z} \map^G(XZ,YZ)
\]
If $X,Y,Z$ are compact there are well known natural ``associativity'' homeomorphisms 
\begin{eqnarray} \label{E:associativity isom XYZ}
&& (XY)Z \xto{\ \ [s[tx,(1-t)y],(1-s)z] \mapsto [stx,s(1-t)y,(1-s)z] \ \ }  XYZ \\
\nonumber
&& X(YZ) \xto{\ [tx,(1-t)[sy,(1-s)z]] \mapsto [tx,s(1-t)y,(1-s)(1-t)z] \ }  XYZ
\end{eqnarray}
This allows us to identify, for example, $\map^G(AY,XY) \xto{J_Z} \map^G((AY)Z,(XY)Z)$ with $\map^G(AY,XY) \xto{J_Z} \map^G(AYZ,XYZ)$.

By inspection, these homeomorphisms together with \eqref{E:JZ explicit formula} imply the commutativity of the following diagrams for compact $G$-spaces $A,T,Y,Z$.
%For any $G$-spaces $A,Y,T$ the functoriality of $J_Z$ implies the commutativity of the  square in the display below, and the commutativity of the triangle is follows by direct calculation using \eqref{E:JZ explicit formula}: $(i^*J_Z)(f)(a)=J_Z(f)([a,0]) =J_Z(f)([1 \cdot a,0]) = [f(a),0] = \incl_*(f)(a)$.
\begin{equation}\label{E:JZ incl}
\vcenter{
\xymatrix{
\map^G(A,T) \ar[r]^{J_Z} \ar[dr]_{\incl_*} & \map^G(AZ,TZ) \ar[d]^{i^*} \\
& \map^G(A,TZ)
}
\xymatrix{
\map^G(AY,T) \ar[r]^{i_Y^*} \ar[d]_{J_Z} &
\map^G(Y,T) \ar[d]^{J_Z} 
\\
\map^G(AYZ,TZ) \ar[r]^{i_{YZ}^*} &
\map^G(YZ,TZ)
}
}
\end{equation}
where $\incl \colon T \xto{t \mapsto [1\cdot t,0z]} TZ$ and $i \colon A \xto{a \mapsto [1 \cdot a,0z]} AZ$ are the inclusions and we used the homeomorphism $(AY)Z \cong AYZ$.

\begin{defn}\label{D:def psiAXYZ}
Let $A,X,Y,Z$ be compact $G$-spaces.
Let
\[
\psi_{A,X,Y,Z} \colon \map^G(A \times X,Y) \to \map^G(A \times XZ,YZ)
\]
be the unique map which renders the following diagram commutative
\[
%\begin{equation}\label{E:def psiAXYZ}
\vcenter{
\xymatrix{
\map^G(A,F(X,Y)) \ar[rr]^{\map^G(A,J_Z)} \ar[d]_{\cong} & &
\map^G(A,F(XZ,YZ)) \ar[d]^{\cong} 
\\
\map^G(A \times X, Y) \ar[rr]_{\psi_{A,X,Y,Z}} & &
\map^G(A \times XZ,YZ). 
}
}
%\end{equation}
\]
\end{defn}

It is clear that $\psi$ is natural in $A$. 
By inspection
\begin{equation}\label{E:psi explicit formula}
\psi_{A,X,Y,Z}(f)(a,[sx,(1-s)z]) = [s \cdot f(a,x),(1-s)z].
\end{equation}

The remainder of this section is devoted to the definition and study of two maps 
\[
\alpha, \beta \colon A \times YZ \times I \to AYZ,
\]
of which the second is simply the quotient onto $A(YZ) \cong AYZ$.
They arise in the computations in Section \ref{Sec:stabilization} in the context of the homeomorphisms \eqref{E:associativity isom XYZ}, and the next definition is our starting point.
Recall that $\Delta^2$ denotes the standard $2$-simplex in $\RR^3$.
For $i=0,1,2$ let $\partial_i \Delta^2$ denote the $i$th face of $\Delta^2$, i.e the elements $(t_0,t_1,t_2) \in \Delta^2$ with $t_i=0$.

\begin{defn}\label{D:def tilde alpha and beta}
Let $\tilde{\alpha}, \tilde{\beta} \colon I \times I \to \Delta^2$ be the functions
\begin{eqnarray*}%\label{E:def tilde alpha and beta}
&& \tilde{\alpha} (s,t) \, = \, \left(st,s(1-t),1-s\right) \\
&& \tilde{\beta}  (s,t) \, = \, \left(t,s(1-t),(1-s)(1-t)\right).
\end{eqnarray*}
\end{defn}

Both maps are clearly surjective, so for any $0 \leq s,t, \leq 1$ there exist $0 \leq s',t' \leq 1$ such that $\tilde{\beta}(s,t)=\tilde{\alpha}(s',t')$.

\begin{prop}\label{P:s' and t'}
Define functions $s',t' \colon I \times I \to I$ as follows.
\begin{eqnarray*}
& s'(s,t) =s+t-st & \\
& t'(s,t) = 
\left\{\begin{array}{rl}
0 & \text{if } s=t=0 \\
\frac{t}{s+t-st} & \text{if } (s,t) \neq (0,0).
\end{array}\right.
&
\end{eqnarray*}
\begin{enumerate}[label=(\alph*)]
\item 
$s'$ is a continuous function, and $t'$ is continuous away from $(0,0)$.
\label{P:s't':cts}

\item
$0 \leq s',t' \leq 1$
\label{P:s't':bounded}

\item 
$\tilde{\beta}=\tilde{\alpha} \circ (s',t')$.
\label{P:s't':alpha-beta}

\item
$s'(0,t)=t$ and $t'(s,0)=0$ and $t'(s,1)=1$. 
Also $t'(0,t)=1$ for all $t \neq 0$.
\label{P:s't':computations}
\end{enumerate}
\end{prop}

\begin{proof}
First, $s+t-st=1-(1-s)(1-t)$.
This shows that $0 \leq s' \leq 1$ and that the denominator in the formula for $t'$ vanishes if and only if $s=t=0$.
This shows that $t'$ is well defined and that it is continuous away from $(0,0)$.
The continuity of $s'$ is clear
Also, $s+t-st=t+s(1-t) \geq t$ which shows that $0 \leq t' \leq 1$.
This proves items \ref{P:s't':cts} and \ref{P:s't':bounded}.
Items \ref{P:s't':alpha-beta} and \ref{P:s't':computations} follow by inspection of the formulas.
\end{proof}

Of course, the maps $\tilde{\alpha}$ and $\tilde{\beta}$ are homotopic for trivial reasons.
But we will need an explicit homotopy satisfying some conditions.
Given a homotopy $h \colon I \times X \to Y$ we will write $h_p \colon X \to Y$ for the restriction of $h$ to $\{p \} \times X$ where $0 \leq p \leq 1$.

\begin{prop}\label{P:tilde theta}
There exists a homotopy $\tilde{\theta} \colon I^2  \times I \to \Delta^2$ from $\tilde{\alpha}$ to $\tilde{\beta}$, written as a family of maps $\tilde{\theta}_p \colon I^2 \to \Delta^2$ parameterized by $0 \leq p \leq 1$,  with the properties
\[
%\begin{equation} %\label{E:tildeTheta at t 0 1}
\begin{array}{lcl}
\tilde{\theta}_p(\{0\} \times I) \subseteq \partial_1 \Delta^2 & , &
\tilde{\theta}_p(\{1\} \times I) \subseteq \partial_2 \Delta^2 
\\
\tilde{\theta}_p(I \times \{0\}) \subseteq \partial_0 \Delta^2 & , &
\tilde{\theta}_p(I \times \{1\}) \subseteq \partial_1 \Delta^2 
\end{array}
%\end{equation}
\]
\end{prop}

\begin{proof}
Define functions $S,T \colon I \times I^2 \to I$ by 
\begin{eqnarray*}
&& T(p,s,t) = p \cdot t'(s,t)+(1-p)t,  \\
&& S(p,s,t)=p \cdot s'(s,t)+(1-p)s.
\end{eqnarray*}
It is clear from Proposition \ref{P:s' and t'}\ref{P:s't':bounded} and \ref{P:s't':cts} that $0 \leq S,T \leq 1$ and that $S$ is continuous and $T$ is continuous away from $I \times \{(0,0)\}$.
Define functions $H,K \colon I \times I^2 \to \Delta^2$ as follows, where we write $S, T$ instead of $S(p,s,t)$ and $T(p,s,t)$, and $t'$ instead of $t'(s,t)$
\begin{eqnarray*}
&& H(p,s,t)=\left(s \cdot T,s\cdot (1-T),1-s\right) \\
&& K(p,s,t)=\left(S\cdot t',S\cdot(1-t'),1-S\right).
\end{eqnarray*}
They are well defined since by Propositions \ref{P:s' and t'}\ref{P:s't':bounded} $0 \leq t' \leq 1$ and we have seen that $0\leq S,T \leq 1$.
They are continuous away from $I \times \{(0,0)\}$ because $S,T$ are and $t'$ is continuous away from $(0,0)$.
% We will write $S_p$ and $T_p$ for $S(p,s,t)$ and $T(p,s,t)$.
Also, it follows from Proposition \ref{P:s' and t'}\ref{P:s't':computations} that
\[
S_p(0,t)=pt \text{ and } 
S_p(1,t)=1 \text{ and }
T_p(s,0)=0 \text{ and }
T_p(s,1)=1.
\]
One then checks that
\[
\begin{array}{lcl}
H_p(\{0\} \times I) \subseteq \partial_1 \Delta^2 & , & 
K_p(\{0\} \times I) \subseteq \partial_1 \Delta^2 
\\
H_p(\{1\}\times I)  \subseteq \partial_2 \Delta^2 & , &
K_p(\{1\} \times I) =\subseteq \partial_2 \Delta^2 
\\
H_p(I \times \{0\})  \subseteq  \partial_0 \Delta^2 &, &  
K_p(I \times \{0\}) \subseteq \partial_0 \Delta^2 
\\
H_p(I \times \{1\})  \subseteq  \partial_1 \Delta^2 & , &
K_p(I \times \{1\})  \subseteq \partial_1 \Delta^2.
\end{array}
\]
Inspection of the definition of $H$ and $K$ gives
\begin{eqnarray*}
&& H_0(s,t)=(st,s(1-t),1-s)=\tilde{\alpha}(s,t) \\
&& H_1(s,t)=(st',s(1-t'),1-s) \\
&& K_0(s,t)=(st',s(1-t'),1-s)= H_1(s,t) \\
&& K_1(s,t) = (s',t',s'(1-t'),1-s') = (\tilde{\alpha} \circ (s',t'))(s,t) = \tilde{\beta}(s,t).
\end{eqnarray*}
Therefore the homotopies $H_p$ and $K_p$ can be concatenated to form a homotopy $\tilde{\Theta} \colon I \times I^2 \to \Delta^2$ from $\tilde{\alpha}$ to $\tilde{\beta}$ with the properties in the statement of this proposition.
\end{proof}

\begin{defn}\label{D:Thetap homotopy}
Let $A,Y,Z$ be {\em compact Hausdorff} spaces and assume that $Y,Z$ are not empty.
Let 
\begin{eqnarray*}
&& q_{A,Y,Z} \colon A \times Y \times Z \times \Delta^2 \to AYZ 
\qquad \text{ and }  \\
&& q_{Y,Z} \colon Y \times Z \times I \to YZ
\end{eqnarray*}
be the restriction of the quotient maps \eqref{E:join quotient}.
Since $Y,Z \neq \emptyset$ the second map is a quotient map.
Compactness of all spaces implies that $A \times q_{Y,Z} \times I$ in the left vertical map in the diagram below is a quotient map too.
It can be described explicitly by the formula
\[
(a,y,z,s,t) \mapsto (a,[sy,(1-s)z],t).
\]
By Proposition \ref{P:tilde theta} and inspection of the formula above, for any $0 \leq p \leq 1$ the composition of the top horizontal arrow with the vertical arrow on the right respects the quotient map on the left.
We finally define $\Theta \colon (A \times YZ \times I) \times I \to AYZ$ to be the homotopy whose fibres $\Theta_p$ are the unique maps which render the following diagram commutative.
\[
\xymatrix{
A \times Y \times Z \times I\times I  \ar[rr]^{A \times Y \times Z \times \tilde{\Theta}_p} \ar[d]_{A \times q_{Y,Z} \times I} & &
A \times Y \times Z \times \Delta^2   \ar[d]^{q_{A,Y,Z}} 
\\
A  \times YZ \times I \ar@{-->}[rr]_{\Theta_p} & &
AYZ 
}
\]
\end{defn}

\begin{defn}\label{D:alpha and beta}
Let $\alpha, \beta \colon A \times YZ \times I \to AYZ$ be the maps $\alpha=\theta|_{p=0}$ and $\beta=\theta|_{p=1}$.
\end{defn}

\begin{prop}\label{P:def alpha and beta via Theta}
The maps $\alpha$ and $\beta$ are homotopic and have the explicit formula
\begin{eqnarray*}
&& \alpha \left(a,[sy,(1-s)z],t\right) =  \left[sta,s(1-t)y,(1-s)z\right]  \\
% \nonumber
&& \beta \left(a,[sy,(1-s)z],t\right) = \left[ta,s(1-t)y,(1-s)(1-t)z\right] 
\end{eqnarray*}
% In particular, $\beta$ is the natural map $A \times YZ \times I \to A(YZ) \cong AYZ$ from \eqref{E:join quotient}.
\end{prop}

\begin{proof}
The homotopy is provided by $\theta_p$.
The formulas are immediate from the explicit description of $\tilde{\alpha}$ and $\tilde{\beta}$ in Definition \ref{D:def tilde alpha and beta} and Proposition \ref{P:tilde theta}.
\end{proof}

The explicit formulas for $\alpha, \beta$ in Proposition \ref{P:def alpha and beta via Theta} give the next straightforward calculation.

\begin{prop}\label{P:restrictions of alpha and beta}
\begin{enumerate}
\item 
The restriction $\alpha|_{A \times YZ \times \{0\}} \colon A\times YZ \to AYZ$ is the composition $A \times YZ \xto{\proj} YZ \xto{\incl} AYZ$.
\label{P:restrictions of alpha and beta:alpha0}

\item
The restriction $\alpha|_{A \times YZ \times \{1\}} \colon A\times YZ \to AYZ$ factors through the inclusion $AZ \subseteq AYZ$ and is given by  $(a,[sy,(1-s)z]) \mapsto [sa,(1-s)z]$.
\label{P:restrictions of alpha and beta:alpha1}

\item
The restriction $\beta|_{A \times YZ \times \{1\}} \colon A\times YZ \to AYZ$  is the composition $A \times YZ \xto{\proj} A \xto{\incl} AYZ$.
\label{P:restrictions of alpha and beta:beta1}
\end{enumerate}
\end{prop}

The following facts are again straightforward calculations:

\begin{prop}\label{P:more on alpha beta restricted}
Let $A,X,Y$ be compact Hausdorff $G$-spaces, $X,Y \neq \emptyset$.
Recall the maps $J_Z$ and $\psi$ from \eqref{E:def JZ} and Definition \ref{D:def psiAXYZ}, and let $\pi \colon A  \times Y \times I \to AY$ be the restriction of the quotient map \eqref{E:join quotient}.
Then the following diagrams commute.
\begin{equation}\label{E:JZ alpha}
\xymatrix{
\map^G(AY,XY) \ar[d]^{J_Z} \ar[r]^{\pi^*} &
\map^G(A  \times Y\times I, XY)  \ar[r]^\cong &
\map^G(A  \times I \times Y, XY) \ar[d]^{\psi_{A \times I,Y,XY,Z}} 
\\
\map^G(AYZ,XYZ) \ar[r]^{\alpha^*} &
% \map^G(A \times I \times Y, XY) \ar[r]_{\psi_{A \times I,Y,XY,Z}} &
\map^G(A \times YZ \times I,XYZ) \ar[r]^\cong &
\map^G(A \times I \times YZ ,XYZ)
}
\end{equation}
\begin{equation}\label{E:JZ alpha0}
\xymatrix{
\map^G(Y,XY) \ar[r]^{J_Z} \ar[d]_{\pi^*} &
\map^G(YZ,XYZ) \ar[d]^{(\alpha|_{A \times YZ \times \{0\}})^*}
\\
\map^G(A \times Y, XY) \ar[r]_{\psi_{A,Y,XY,Z}} &
\map^G(A \times YZ,XYZ)
}
\end{equation}
\begin{equation}\label{E:JZ alpha1}
\xymatrix{
\map^G(A,XY) \ar[r]^{J_Z} \ar[d]_{\pi^*} &
\map^G(AZ,XYZ) \ar[d]^{(\alpha|_{A \times YZ \times \{1\}})^*}
\\
\map^G(A \times Y, XY) \ar[r]_{\psi_{A,Y,XY,Z}} &
\map^G(A \times YZ,XYZ)
}
\end{equation}
\begin{equation}\label{E:JZ beta1}
\xymatrix{
\map^G(A,XY) \ar[r]^{J_Z} \ar[d]_{\incl_*} &
\map^G(AZ,XYZ) \ar[d]^{(\beta|_{A \times YZ \times \{1\}})^*} \ar[dl]^{i^*}
\\
\map^G(A,XY) \ar[r]_{\pi^*} &
\map^G(A \times YZ,XYZ)
}
\end{equation}
%In addition $\alpha|_{A \times YZ \times \{0\}} \colon A \times YZ \to YZ$ is the projection, $\alpha|_{A \times YZ \times \{1\}} \colon A \times YZ \to AZ$ is the map $(a,[sy,(1-s)z]) \mapsto [sa,(1-s)z]$ and $\beta|_{A \times YZ \times \{1\}}$ is the composition $A \times YZ \xto{\pi_A} A \xto{i_A^{AZ}} AZ$.
\end{prop}

%
% Proof.
\begin{proof}
To check \eqref{E:JZ alpha} we use the formula for $\alpha$ in Proposition \ref{P:def alpha and beta via Theta}, for $J_Z$ in \eqref{E:JZ explicit formula} and for $\psi$ in \eqref{E:psi explicit formula}, to calculate
\begin{multline*}
\alpha^* J_Z(f) (a,t,[sy,(1-s)z]) = 
J_Z(f)([st\cdot a,s(1-t)\cdot y,(1-s)\cdot z]) =
[s\cdot f([ta,(1-t)y]),(1-s)z] =
\\
[s\cdot (\pi^* f)(a,t,y),(1-s)z]=
\psi_{A \times I,Y,XY,Z}(\pi^* f)(a,t,[sy,(1-s)z]).
\end{multline*}
The commutativity of \eqref{E:JZ alpha0} follows from Proposition \ref{P:restrictions of alpha and beta}\ref{P:restrictions of alpha and beta:alpha0} , the naturality of $\psi$ with respect to $A \to *$, and the observation that $J_Z$ is $\psi_{*,Y,XY,Z}$. % \colon \map^G(* \times Y) \to \map^G(* \times YZ,XYZ)$.
%
%we use the naturality of $\psi$ with respect to $A \to *$ and the fact that $\alpha_0 \colon A \times YZ \to YZ$ is the projection and that $J_Z$ is $\psi_{*,Y,XY,Z} \colon \map^G(* \times Y) \to \map^G(* \times YZ,XYZ)$.

The commutativity of \eqref{E:JZ alpha1} follows by the following calculation which uses Proposition \ref{P:restrictions of alpha and beta}\ref{P:restrictions of alpha and beta:alpha1}, and equations \eqref{E:JZ explicit formula} and \eqref{E:psi explicit formula}.
\begin{multline*}
(\alpha_1^* J_Z(h))(a,t,[sy,(1-s)z]) = 
J_Z(h)([sa,[1-s)z]) =
\\
[s\cdot h(a),(1-s)z] = 
[s \cdot (\pi^* h)(a,y),(1-s)z] =
\psi_{A,Y,XY,Z}(\pi^* h)(a,[sy,(1-s)z]).
\end{multline*}
Finally, \eqref{E:JZ beta1} follows from \eqref{E:JZ incl} and Proposition \ref{P:restrictions of alpha and beta}\ref{P:restrictions of alpha and beta:beta1}.
\end{proof}

%
% ---
%
\section{Filtration of $G$-spaces}
\label{S:filtration}

%\begin{void}\label{V:filtration}
Let $G$ be a finite group.
Let $(H)$ denote the conjugacy class of $H \leq G$.
Enumerate the conjugacy classes of $G$
\begin{equation}\label{E:enumerate subgroups}
(H_1), \cdots , (H_r)
\end{equation}
so that $|H_i| \geq |H_{i+1}|$.
In this way, if $H_i$ is conjugate to a proper subgroup of $H_j$ then $i>j$.
%  particular $\map^G(G/H_i,G/H_j)$ is not empty only if $i>j$.

Let $X$ be a $G$-space.
Let $G_x$ denote the isotropy group of $x \in X$.
For any $0 \leq q \leq r$ set
\begin{equation}\label{E:def filtration X_q}
X_q = \{ x \in X : G_x \in (H_i) \text{ for some } 1 \leq i \leq q\}.
\end{equation}
We obtain a filtration of $X$, see \cite[Chap. I.6]{tomDieck87},
\[
\emptyset=X_0 \subseteq X_1 \subseteq \cdots \subseteq X_r=X.
\]
If $X$ is a $G$-CW complex then one checks that $X_q$ are subcomplexes \cite[Prop. II.1.12]{tomDieck87}.
The assignment $X \mapsto X_q$ is a functor giving rise to natural maps 
\[
\map^G(X,Y) \xto{ \ f \mapsto f|_{X_q} \ } \map^G(X_q,Y_q).
\]
This is because $f(X_q) \subseteq Y_q$ by the choice of the enumeration \eqref{E:enumerate subgroups}, see \cite[I.(6.3)]{tomDieck87}.

\begin{prop}\label{P:filtration facts G}
Let $X,Y$ be compact Hausdorff $G$-spaces.
\begin{enumerate}
\item
\label{P:filtration facts G:join-H}
If $H \leq G$ then $(XY)^H =X^HY^H$ (where $XY$ denotes the join).

\item
\label{P:filtration facts G:join-q}
For any $1 \leq q \leq r$ 
\[
(XY)_q \subseteq X_q Y_q \subseteq X_q Y \subseteq XY
\]

\item
\label{P:filtration facts G:fix H}
Set $H=H_q$ for some $1 \leq q \le r$.
Then
\begin{eqnarray*}
&& X_q^H=X^H \\
&& X_{q-1}^H  = \bigcup_{K\gneq H} X^K.
\end{eqnarray*}
\iffalse
\item
\label{P:filtration facts G:codim}
Let $X$ be a finite dimensional $G$-CW complex and $d \geq 0$.
Let $H=H_q$ for some  $1 \leq q \leq r$ and assume that  $H \in \Iso_G(X)$.
Assume further that $\dim X^H - \dim X^K \geq d$ for any $K \gneq H$.
Then $\dim X_q - \dim X_{q-1} \geq d$ and $\dim X^H-\dim X_{q-1}^H \geq d$.
\fi
\end{enumerate}
\end{prop}

\begin{proof}
Notice that the isotropy group of $[x,y,t] \in XY$ is $G_x$ if $t=1$ and $G_y$ if $t=0$ and $G_x \cap G_y$ if $0<t<1$.
From this items \ref{P:filtration facts G:join-H}  and \ref{P:filtration facts G:join-q} follow easily.

Set $H=H_q$ for some $1 \leq q \leq r$.
By construction 
\[
X_q \setminus X_{q-1} = \{x \in X :  G_x \in (H) \}.
\]
The choice of the enumeration \eqref{E:enumerate subgroups} implies item \ref{P:filtration facts G:fix H} since $x \in X^H$ if and only if $H \leq G_x$.
% Item \ref{P:filtration facts G:codim} is an immediate consequence.
\end{proof}

Recall that if $X$ is a $G$-space and $H \leq G$ then $X^H$ admits an action of $WH=N_GH/H$.
This gives a functor $X \mapsto X^H$ from $G$-spaces to $WH$-spaces. 
There results a natural map
\begin{equation}\label{E:def res_G^H}
\map^G(X,Y) \xto{\ \res_G^H \colon f \mapsto f|_{X^H} \ } \map^{WH}(X^H,Y^H)
\end{equation}

\begin{prop}\label{P:JZ and WH}
Let $X,Y,Z$ be compact Hausdorff $G$-spaces.
\begin{enumerate}
\item
\label{P:JZ and WH:res_G^H and JZ} %\label{E:res_G^H and JZ}
The join map $J_Z$ \eqref{E:def JZ} renders the following square commutative
\[
\xymatrix{
\map^G(X,Y) \ar[r]^{\res_G^H} \ar[d]_{J_Z} &
\map^{WH}(X^H,Y^H) \ar[d]^{J_{Z^H}} 
\\
\map^G(XZ,YZ) \ar[r]_(0.4){\res_G^H} & \map^{WH}((XZ)^H,(YZ)^H).
}
\]

\item
\label{P:JZ and WH:free action of WH on X-X'} %\label{E:free action of WH on X-X'}
Let $H \leq G$.
Then
\[
\Iso_{WH}(X^H,X_{q-1}^H) \subseteq \{e\}.
\]
\end{enumerate}
\end{prop}

\begin{proof}
Item \ref{P:JZ and WH:res_G^H and JZ} follows from Proposition \ref{P:filtration facts G}\ref{P:filtration facts G:join-H} and by inspection of \eqref{E:JZ explicit formula}.
For item \ref{P:JZ and WH:free action of WH on X-X'},  suppose that $x \in X^H \setminus X_{q-1}^H$.
Then $H \leq G_x$ and by choice of the enumeration \eqref{E:enumerate subgroups}, $G_x \in (H_i)$ for some $i \leq q$.
Since $x \notin X_{q-1}$ it follows that $i=q$ and therefore $G_x=H$.
In particular, $WH_x$ is trivial.
\end{proof}

\begin{prop}\label{P:isom on fibres in F_q filration}
Let $X$ be a $G$-CW complex.
Set $H=H_q \in \Iso_G(X)$ for some $1 \leq q \leq r$.
Then there is a pullback square, natural in both $X$ and $Y$
\[
\xymatrix{
\map^G(X_q,Y) \ar@{->>}[r]^{i^*} \ar[d]_{\res_G^H} & \map^G(X_{q-1},Y) \ar[d]^{\res_G^H}  \\
\map^{WH}(X^H,Y^H) \ar@{->>}[r]^{i^*} & \map^{WH}(X_{q-1}^H,Y^H)
}
\]
whose rows are fibrations, hence the vertical arrows induce homeomorphisms on all fibres.
\end{prop}

\begin{proof}
Since $\Iso_G(X_q,X_{q-1})$ is the conjugacy class of $H=H_q$, there is a pushout square %\todo{Reference}
%\[
\begin{equation}\label{E:NGH pushout square} 
\xymatrix{
G \times_{NH} X_{q-1}^H \ar[rr]^(0.55){(g,x) \mapsto gx } \ar@{^(->}[d] & & X_{q-1} \ar@{^(->}[d] 
\\
G \times_{NH} X_{q}^H \ar[rr]^(0.55){(g,x) \mapsto gx } & & X_{q}
}
\end{equation}
%\]
in which the vertical arrows are inclusions of $G$-CW complexes, i.e $G$-cofibrations, and which is is natural in $X$.
The pullback square is obtained upon applying the functor $\map^G(-,Y)$ to this pushout square and observing that $X_q^H=X^H$ (Proposition \ref{P:filtration facts G}\ref{P:filtration facts G:fix H}) and that if $A$ is a $WH$-space then there are natural homeomorphisms 
\[
\map^G(G \times_{NH} A,Y) \cong \map^{NH}(A,Y)=\map^{NH}(A,Y^H) = \map^{WH}(A,Y^H).
\]
%and 
%Recall that if $Y$ is a $G$-space and $A$ is a $K$-space for some $K \leq G$ then there is a natural homeomorphism $\map^G(G \times_K A,Y) \cong \map^K(A,Y)$.
%Also, if $N \nsg K$ fixes both $A$ and $Y$ then $K/N$ acts on both $A$ and $Y$ and $\map^K(A,Y) =\map^{K/N}(A,Y)$.
%The pullback square in the statement of the proposition is obtained by applying the functor $\map^G(-,Y)$ -- which is natural in $Y$ -- to the pushout diagram above, and using $X_q^H=X^H$ (Proposition \ref{P:filtration facts G}\ref{P:filtration facts G:fix H}).
\end{proof}

%
% ---
%
\section{The stabilization lemma}
\label{Sec:stabilization}

The purpose of this section is to prove the following Proposition.

\begin{prop}\label{P:stabilisation main result}
Let $G$ be a finite group.
Let $X,Y,Z$ be finite $G$-CW complexes. Let $k \geq 0$.
Assume that
\begin{enumerate}
\item 
$\Iso_G(X)=\Iso_G(Y)=\Iso_G(XYZ)$,
\label{P:stab hypothesis 1 v2}
\end{enumerate}
and that for any $H \in \Iso_G(X)$
\begin{enumerate}
\setcounter{enumi}{1}
\item
$\dim Y^H >k$,
\label{P:stab hypothesis 2 v2}

\item
$\dim X^H - \dim \left(\bigcup_{K \gneq H} X^K\right) > k$. % for any $K \gneq H$,
\label{P:stab hypothesis 3 v2}

\item
$\conn (XY)^H \geq \dim X^H+ \dim Y^H$ and $\conn (XYZ)^H \geq \dim X^H+ \dim (YZ)^H$
\label{P:stab hypothesis 5 v2}

\item
$F(Y^H,(XY)^H) \xto{J_{Z^H}} F((YZ)^H, (XYZ)^H)$ is a non-equivariant $(\dim X^H+k+1)$-equivalence (see \eqref{E:def JZ} and Proposition \ref{P:filtration facts G}\ref{P:filtration facts G:join-H}).
\label{P:stab hypothesis 4 v2}
\end{enumerate}
Then the natural map
\[
\map^G(XY,XY) \xto{J_Z} \map^G(XYZ,XYZ)
\]
is a $k$-equivalence.
\end{prop}

\begin{proof}
We will use the filtration \eqref{E:def filtration X_q} and show that the composition
\begin{equation}\label{E:stabilisation induction 1}
\map^G((XY)_q,XY) \xto{\ J_Z \ } \map^G((XY)_qZ,XYZ) \xto{j^*} \map^G((XYZ)_q,XYZ)
\end{equation}
is a $k$-equivalence for any $0 \leq q \leq r$, where $j$ denotes the inclusion $(XYZ)_q \subseteq (XY)_qZ$, see Proposition \ref{P:filtration facts G}\ref{P:filtration facts G:join-q}.
The claim of the proposition follows for $q=r$.

The proof is by induction on $0 \leq q \leq r$.
The base of induction $q=0$ is a triviality since $(XY)_q=(XYZ)_q=\emptyset$.
We therefore assume that \eqref{E:stabilisation induction 1} is a $k$-equivalence for $q-1$ and we prove it for $q \leq r$.
Set 
\[
H=H_q.
\]
Since $X \subseteq XY \subseteq XYZ$, it follows from hypothesis \ref{P:stab hypothesis 1 v2} that $\Iso_G(X)=\Iso_G(XY)=\Iso_G(XYZ)$.
By the definition of the filtration \eqref{E:def filtration X_q}, if $H \notin \Iso_G(X)$ then $(XY)q=(XY)_{q-1}$ and $(XYZ)_q=(XYZ)_{q-1}$, in which case the induction step follows trivially from its hypothesis.
We therefore assume that
\[
H \in \Iso_G(X).
\]
Naturality of $J_Z$ gives the following commutative diagram whose rows are fibrations and in which $i$ denotes the inclusions $(XY)_{q-1} \subseteq (XY)_q$ and $(XYZ)_{q-1} \subseteq (XYZ)_q$, and $j$ denotes the inclusions $(XYZ)_q \subseteq (XY)_qZ$ and $(XYZ)_{q-1} \subseteq (XY)_{q-1}Z$. 
\begin{equation}\label{E:stabilisation 2}
\vcenter{
\xymatrix{
\map^G((XY)_q,XY) \ar@{->>}[r]^{i^*} \ar[d]_{j^* \circ J_Z} &
\map^G((XY)_{q-1},XY) \ar[d]^{j^* \circ J_Z} 
\\
\map^G((XYZ)_q,XYZ) \ar@{->>}[r]^{i^*} &
\map^G((XYZ)_{q-1},XYZ).
}
}
\end{equation}
The vertical arrow on the right is a $k$-equivalence by the induction hypothesis, so by Lemma \ref{L:k equivalences two out of three}\ref{L:k equivalences 2 outof 3:total} it remains to show that the fibres of $i^*$ are $k$-equivalent.
% vertical arrows in this diagram induce a $k$-equivalence on all the fibres of the fibrations $i^*$.

Proposition \ref{P:filtration facts G}\ref{P:filtration facts G:join-H} and the formula \eqref{E:JZ explicit formula} for $J_Z(f)$ easily imply the commutativity of the following diagram, where $i$ and $\ell$ denote the inclusions of the $H$-fixed points of $(XY)_{q-1} \subseteq (X_{q-1}Y)_q \subseteq (XY)_q$ and $(XYZ)_{q-1} \subseteq (X_{q-1}YZ)_q \subseteq (XYZ)_q$, and $j$ denoted the inclusion of the $H$-fixed points of $(XYZ)_{q-1} \subseteq (XY)_{q-1}Z$.
\begin{equation}\label{E:stabilisation 3}
\vcenter{
\xymatrix{
\map^{WH}((XY)^H,(XY)^H) \ar@{->>}[r]^{i^*} \ar@{=}[d] \ar@/_1pc/@<-14ex>[ddd]_{J_{Z^H}} &
\map^{WH}((X_{q-1}Y)^H,(XY)^H) \ar@{->>}[d]^{\ell^*} \ar@/^1pc/@<15ex>[ddd]^{J_{Z^H}}
\\
\map^{WH}((XY)^H,(XY)^H) \ar@{->>}[r]^{i^*} \ar[d]_{J_{Z^H}} &
\map^{WH}((XY)_{q-1}^H,(XY)^H) \ar[d]_{j^* \circ J_{Z^H}} 
\\
\map^{WH}((XYZ)^H,(XYZ)^H) \ar@{->>}[r]^{i^*} \ar@{=}[d] 
&
\map^{WH}((XYZ)_{q-1}^H,(XYZ)^H) 
\\
\map^{WH}((XYZ)^H,(XYZ)^H) \ar@{->>}[r]^{i^*} 
&
\map^{WH}((X_{q-1}YZ)^H,(XYZ)^H)  \ar@{->>}[u]_{\ell^*}
}
}
\end{equation}
By Proposition \ref{P:JZ and WH}\ref{P:JZ and WH:res_G^H and JZ}, the maps $\res_G^H$ in \eqref{E:def res_G^H} give rise to a natural transformation between the commutative square \eqref{E:stabilisation 2} to the square in the middle of \eqref{E:stabilisation 3}.
By Proposition \ref{P:isom on fibres in F_q filration} the fibres of the rows of \eqref{E:stabilisation 2} over any $f \in \map^G((XY)_{q-1},XY)$ are homeomorphic to the fibres  over $\res_G^H(f)$ of the rows of the middle square of \eqref{E:stabilisation 3}.

Therefore, it suffices to prove that the fibres of the rows of the 2nd square in \eqref{E:stabilisation 3} are $k$-equivalent for any choice of basepoint in $\map^{WH}((XY)_{q-1}^H,(XY)^H)$.
If $Z^H=\emptyset$ then this is a triviality since $J_{Z^H}$ and $j^*$ are the identity maps.
So for the remainder of the proof we assume that
\[
Z^H \neq \emptyset.
\]
It follows from hypothesis \ref{P:stab hypothesis 3 v2} and from Proposition \ref{P:filtration facts G}\ref{P:filtration facts G:fix H} %\eqref{E:X_q^H properties} 
that $\dim X^H-\dim X_{q-1}^H \geq k+1$.
Together with hypotheses \ref{P:stab hypothesis 5 v2} and \ref{P:stab hypothesis 2 v2} we get
\[
\conn (XY)^H - \dim (X_{q-1}Y)^H \geq \dim X^H+ \dim Y^H - (\dim X_{q-1}^H+\dim Y^H+1) \geq k.
\]
Similarly, 
\[\conn (XYZ)^H \geq \dim (X_{q-1}YZ)^H+k.
\]
Since $H \in \Iso_G(X)$ Proposition \ref{P:JZ and WH}\ref{P:JZ and WH:free action of WH on X-X'}  implies that $\Iso_{WH}((X_{q-1}Y)^H,(XY)_{q-1}^H) =\{e\}$.
% Notice that $\Iso_{WH}((X_{q-1}Y)^H,(XY)_{q-1}^H) =\{e\}$ by Proposition \ref{P:JZ and WH}\ref{P:JZ and WH:free action of WH on X-X'}  since $H \in \Iso_G(X)$ so $X_{q-1}^H \neq X^H$.
Corollary \ref{C:maps with highly connected target Y} applies with $(XY)_{q-1}^H \subseteq (X_{q-1}Y)^H$ and with $(XYZ)_{q-1}^H \subseteq (X_{q-1}YZ)^H$ to show that both maps $\ell^*$ in \eqref{E:stabilisation 3} are $k$-equivalences.
It follows from Lemma \ref{L:k equivalences two out of three}\ref{L:k equivalences 2 outof 3:fibres} that the fibres of $i^*$ at the top and bottom squares of \eqref{E:stabilisation 3} are $k$-equivalent.
Since $\ell^*$ are bijective on components, it suffices to show that the fibres of $i^*$ at the top and bottom of \eqref{E:stabilisation 3} are $k$-equivalent via the curved arrows.

This is indeed the case by applying Proposition \ref{P:stabilisation key to WH reduction} below with $X_{q-1}^H \subseteq X^H$ and $Y^H$ and $Z^H$ and $G=WH$.
To see this, notice first that $X^H$ and $Y^H$ are not empty since $H \in \Iso_G(X)=\Iso_G(Y)$.
Also $Z^H \neq \emptyset$ by assumption.
Hypothesis \ref{P:stabilisation key to WH reduction hyp 0} of Proposition \ref{P:stabilisation key to WH reduction} follows from Proposition \ref{P:JZ and WH}\ref{P:JZ and WH:free action of WH on X-X'}.
Hypothesis \ref{P:stabilisation key to WH reduction hyp 2} of Proposition \ref{P:stabilisation key to WH reduction} is hypothesis \ref{P:stab hypothesis 4 v2} of this proposition.
Hypothesis \ref{P:stabilisation key to WH reduction hyp 4} of Proposition \ref{P:stabilisation key to WH reduction} follows from hypotheses \ref{P:stab hypothesis 5 v2} and \ref{P:stab hypothesis 2 v2} of this proposition.
\end{proof}

%
% ---
%c
\begin{prop}\label{P:stabilisation key to WH reduction}
Let $G$ be a finite group and $X,Y,Z$ be finite non-empty $G$-CW complexes and $X' \subseteq X$ a $G$-subcomplex.
Let $k \geq 0$.
Suppose that % $X' \subseteq X$ is a subcomplex such that $\Iso_G(X,X') = \{e\}$.
\begin{enumerate}
\item
\label{P:stabilisation key to WH reduction hyp 0}
$\Iso_G(X,X') = \{e\}$.

%\item
%\label{P:stabilisation key to WH reduction hyp 1}
% $\dim X - \dim X' > k+1$ and $\dim Y>k+1$.
%$\dim Y>k+1$.

\item
\label{P:stabilisation key to WH reduction hyp 2}
$F(Y,XY) \xto{J_Z} F(YZ,XYZ)$ %(non-equivariant function complex) 
is a {\em non-equivariant} $(\dim X+k+1)$-equivalence. 

%\item
%\label{P:stabilisation key to WH reduction hyp 3}
%The inclusion $XY \to XYZ$ is a non-equivariant $\dim X+ \dim Y$ equivalence.

\item
\label{P:stabilisation key to WH reduction hyp 4}
$\conn(XY) \geq \dim X +k +1$.    

\end{enumerate}
Then the maps induced on fibres of the horizontal arrows in the following diagram
\[
\xymatrix{
\map^G(XY,XY) \ar@{->>}[r]^{i^*} \ar[d]_{J_Z} & \map^G(X'Y,XY) \ar[d]^{J_Z} 
\\
\map^G(XYZ,XYZ) \ar@{->>}[r]^{i^*}  & \map^G(X'YZ,XYZ)  
}
\]
are $k$-equivalences for any choice of basepoint in the space at the top right corner.
\end{prop}

% \comment{To do in section on joins etc. I have switched from $X \times I \times Y \to XY$ to $X \times Y \times I \to XY$}

\begin{proof}
Let $i \colon X' \to X$ be the inclusion.
Define the following objects in $\catsq$, see Definition \ref{D:cat sq}.
%\begin{eqnarray*}
\[
%\A &=& 
\A =
\vcenter{
\xymatrix{
\map^G(XY,XY) \ar@{->>}[r]^(0.35){(i_Y^*,i_X^*)} \ar@{->>}[d]_{i^*} &
\map^{G}(Y,XY) \times \map^G(X,XY) \ar@{->>}[d]^{\id \times i^*} 
\\
\map^{G}(X'Y,XY) \ar@{->>}[r]^(0.37){(i_Y^*,i_{X'}^*)}  &
\map^{G}(Y,XY) \times \map^{G}(X',XY).
}
}
\]
%\begin{eqnarray*}
%\\
\[
% \B &=& 
\B =
\vcenter{
\xymatrix{
\map^{G}(XYZ,XYZ) \ar@{->>}[r]^(0.4){(i_{YZ}^*,i_X^*)} \ar@{->>}[d]_{i^*} &
\map^{G}(YZ,XYZ) \times \map^{G}(X,XYZ) \ar@{->>}[d]^{\id \times i^*} 
\\
\map^{G}(X'YZ,XYZ) \ar@{->>}[r]^(0.4){(i_{YZ}^*,i_{X'}^*)}  &
\map^{G}(YZ,XYZ) \times \map^{G}(X',XYZ).
}
}
\]
%\\
\[
%\C &=&
\C = 
\vcenter{
\xymatrix{
\map^{G}(X \times Y \times I,XY) \ar@{->>}[r]^{(\ev_0,\ev_1)} \ar@{->>}[d]_{i^*} &
\map^{G}(X \times Y,XY)^2 \ar@{->>}[d]^{i^* \times i^*} 
\\
\map^{G}(X' \times Y \times I,XY) \ar@{->>}[r]^{(\ev_0,\ev_1)} &
\map^{G}(X' \times Y,XY)^2.
}
}
\]
%\\
\[
% \D &=&
\D =
\vcenter{
\vbox{
\xymatrix{
\map^G\left(X \times YZ \times I,XYZ\right) \ar@{->>}[r]^{(\ev_0,\ev_1)} \ar@{->>}[d]_{i^*} &
\map^{G}(X \times YZ,XYZ)^2 \ar@{->>}[d]^{i^*\times i^*}
\\
\map^{G}\left(X' \times YZ \times I,XYZ\right) \ar@{->>}[r]^(0.53){(\ev_0,\ev_1)} &
\map^{G}(X' \times YZ,XYZ)^2
}
}
}
\]
%\end{eqnarray*}
The commutativity of these squares is a direct consequence of the naturality of $i \mapsto i^*$.
The plan of the proof is as follows.
\begin{itemize}
\item[(a)]
Define morphisms $\A \xto{\Phi} \B \xto{\Pi} \D$ and $\A \xto{\Pi} \C \xto{\Psi} \D$ in $\catsq$.
{\em Note:} We used $\Pi$ to denote two different morphisms; This will create no source of confusion and the reason for the choice will become apparent in \eqref{E:def Pi B to D} and \eqref{E:def Pi A to C} where they are defined.
\item[(b)]
Show that both $\Lambda(\B,\ul{y}) \xto{\Lambda(\Pi)} \Lambda(\D,\Pi(\ul{y}))$ and $\Lambda(\A,\ul{x}) \xto{\Lambda(\Pi)} \Lambda(\C,\Pi(\ul{x}))$ are weak homotopy equivalences for any choice of basepoints $\ul{x}$ for $\A$ and $\ul{y}$ for $\B$, see Definition \ref{D:cat sq} and equation \eqref{E:def Lambda}.
\item[(c)]
Show that $\Lambda(\C,\ul{y}) \xto{\Lambda(\Psi)} \Lambda(\D,\Psi(\ul{y}))$ is a $k$-equivalence for any choice of basepoint $\ul{y}$ in $\C$.
\item[(d)]
Show that $\A \xto{\Pi \circ \Phi} \D$ is homotopic to $\A \xto{\Psi \circ \Pi} \D$ (Definition \ref{D:homotopy in catsq}).
\item[(e)] Deduce that $\Lambda(\A,\ul{x}) \xto{\Lambda(\Phi)} \Lambda(\B,\Phi(\ul{x}))$ is a $k$-equivalence for any choice of base point $\ul{x}$ for $\A$.
Use this and Lemma \ref{L:elementary facts on Lambda}\ref{L:elementary facts on Lambda 4}  to complete the proof.
\end{itemize}
With the indexing in \eqref{E:object of catsq}, we will now describe maps $\Phi_i \colon \A_i \to \B_i$ and $\Pi_i \colon \B_i \to \D_i$ and $\Pi_i \colon \A_i \to \C_i$ and $\Psi_i \colon \C_i \to \D_i$, where $i=0,\dots,3$.
We will then show that these are the components of natural transformations $\A \xto{\Phi} \B$ and $\B \xto{\Pi} \D$ and $\A \xto{\Pi} \C$ and $\C \xto{\Psi} \D$.

\noindent
{\em Notation:}
In what follows we will use the letter $A$ to represents either $X$ or $X'$,
\[
A=X' \text{  or  } A=X.
\]
Define $\Phi_i \colon \A_i \to \B_i$  using the maps $J_Z$ \eqref{E:def JZ} and the inclusion $\incl \colon XY \to XYZ$ as follows, %\todo{Comment on $\incl$ in section on joins}
\begin{eqnarray}
\label{E:def Phi}
&& \Phi_1,\Phi_3 \colon \map^{G}(AY,XY) \xto{J_Z} \map^{G}(AYZ,XYZ) \\
\nonumber
&& \Phi_0, \Phi_2  \colon \map^{G}(Y,XY) \times \map^{G}(A,XY) \xto{J_Z \times \incl_*} \map^{G}(YZ,XYZ) \times \map^{G}(A,XYZ)
\end{eqnarray}
Let $\pi \colon A \times YZ  \times I \to AYZ$ be the map in \eqref{E:join pushout}. %\todo{Changes of order of $I$}
Let $\pi_A \colon A \times YZ \to A$ and $\pi_{YZ} \colon A \times YZ \to YZ$ be the projections.
Define $\Pi_i \colon \B_i \to \D_i$ as follows,
\begin{eqnarray}\label{E:def Pi B to D}
&& \Pi_1, \Pi_3 \colon \map^{G}(AYZ,XYZ) \xto{\ \pi^*\ } \map^{G}(A \times YZ \times I,XYZ) \\
\nonumber %\label{E:def PI B to D at 2 0}
&& \Pi_0, \Pi_2 \colon \map^{G}(YZ,XYZ) \times \map^{G}(A,XYZ) \xto{\pi_{YZ}^* \times \pi_A^*} \map^{G}(A \times YZ,XYZ)^2 
\end{eqnarray}
Let $\pi \colon A \times Y \times I \to AY$ be as in \eqref{E:join pushout} and $\pi_A \colon A \times Y \to A$ and $\pi_Y \colon A \times Y \to Y$ be the projections.
Set  $\Pi_i \colon \A_i \to C_i$ as follows
\begin{eqnarray}\label{E:def Pi A to C}
&& \Pi_3, \Pi_1 \colon \map^{G}(AY,XY) \xto{\pi^*} \map^{G}(A \times Y \times I, XY) \\
\nonumber
&& \Pi_2, \Pi_0 \colon \map^{G}(Y,XY) \times \map^{G}(A,XY) \xto{\pi_Y^* \times \pi_A^*} \map^{G}(A \times Y, XY)^2 
\end{eqnarray}
Use the maps $\psi_{A\times I,Y,XY,Z}$ and $\psi_{A,Y,XY,Z}$ in Definition \ref{D:def psiAXYZ} to define $\Psi_i \colon \C_i \to \D_i$
\begin{eqnarray}\label{E:def Psi}
&& \Psi_3,\Psi_1 \colon \map^{G}(A \times Y \times I,XY) \xto{\psi} \map^{G}(A \times YZ \times I,XYZ) \\
\nonumber
&& \Psi_2,\Psi_0 \colon \map^{G}(A \times Y,XY)^2 \xto{\psi \times \psi} \map^{G}(A \times YZ,XYZ)^2 
\end{eqnarray}
\iffalse   %%%%%% Description of \Psi : \C \to \D. NOT NEEDED (START)
Thus, for any $f \in \map^G(A \times I \times Y,XY)$ and any $(g,h) \in \map^G(A \times Y, XY)^2$
\begin{eqnarray*}
\Psi_1(f), \Psi_3(f) &\colon& (a,t,[sy,\bar{s}z])  \mapsto [s\cdot f(a,t,y),\bar{s}z] \\
\Psi_2(g), \Psi_0(g) &\colon& (a,[sy,\bar{s}z])  \mapsto [s\cdot g(a,y),\bar{s}z] \qquad \text{(and similarly for $h$)}
\end{eqnarray*}
\fi   %%%%%% Description of \Pi : \C \to \D. NOT NEEDED (END)
% We now check that these give natural transformations in $\catsq$.

\noindent
{\em Claim 1:} The maps $\Phi_0,\dots,\Phi_3$ in \eqref{E:def Phi} define a natural transformation $\Phi \colon \A \to \B$.

\noindent
Proof: With the indexing of Definition \ref{D:cat sq},  the naturality of $J_Z$ and the equality $\incl_* \circ i^* = i^* \circ \incl_*$ imply that  $\Phi_1 \circ a_{31}=b_{31} \circ \Phi_3$  and $\Phi_0 \circ a_{20}=b_{20} \circ \Phi_2$.
The commutative square in \eqref{E:JZ incl} implies the commutativity of the following diagram
\[
\xymatrix{
\map^G(AY,XY) \ar[rr]^(0.4){((i_Y^{AY})^*,(i_A^{AY})^*)} \ar[d]_{J_Z} &  &
\map^G(Y,XY) \times \map^G(A,XY) \ar[d]^{J_Z \times J_Z} 
\\
\map^G(AYZ,XYZ) \ar[rr]^(0.4){((i_{YZ}^{AYZ})^*,(i_{AZ}^{AYZ})^*)} & &
\map^G(YZ,XYZ) \times \map^G(AZ,XYZ)
}
%\quad
%\xymatrix{
%\map^G(AY,XY) \ar[r]^{(i_Y^{AY})^*} \ar[d]_{J_Z} & 
%\map^G(Y,XY) \ar[d]^{J_Z} \\
%\map^G(AYZ,XYZ) \ar[r]^{(i_{YZ}^{XYZ})^*} &
%\map^G(YZ,XYZ)
%}
\]
%\[
%\xymatrix{
%\map^G(AY,XY) \ar[r]^{(i_A^{AY})^*} \ar[d]_{J_Z} &  
%\map^G(A,XY) \ar[d]^{J_Z} \\
%\map^G(AYZ,XYZ) \ar[r]^{(i_{AZ}^{AYZ})^*} &
%\map^G(AZ,XYZ)
%}
%\quad
%\xymatrix{
%\map^G(AY,XY) \ar[r]^{(i_Y^{AY})^*} \ar[d]_{J_Z} & 
%\map^G(Y,XY) \ar[d]^{J_Z} \\
%\map^G(AYZ,XYZ) \ar[r]^{(i_{YZ}^{XYZ})^*} &
%\map^G(YZ,XYZ)
%}
%\]
Composing the 2nd factor of the 2nd column with $\map^G(AZ,XYZ) \xto{i_A^*} \map^G(A,XYZ)$ and using the commutative triangle in \eqref{E:JZ incl}, it follows that $\Phi_2 \circ a_{32}=b_{32} \circ \Phi_3$ and $\Phi_0 \circ a_{10}=b_{10} \circ \Phi_1$.
Hence, $\Phi \colon \A \to \B$ is a morphism in $\catsq$.
%Compose the last of the first diagram with $\map^G(AZ,XYZ) \xto{(i_A^{AZ})^*} \map^G(A,XYZ)$ and use \eqref{E:JZ incl} to deduce that $\Phi_2 \circ a_{32}=b_{32} \circ \Phi_3$ and $\Phi_0 \circ a_{10}=b_{10} \circ \Phi_1$.
%This completes the proof that $\Phi \colon \A \to \B$ is a morphism in $\catsq$.
\hfill QED

\noindent
{\em Claim 2:} The maps $\Psi_0,\dots,\Psi_3$ in \eqref{E:def Psi} define a morphism $\Psi \colon \C \to \D$.

\noindent
{\em Proof:} This is immediate from the naturality of $\psi$ with respect to the inclusions $X' \subseteq X$ and $X' \times I \subseteq X \times I$ and the inclusions $A \coprod A \subseteq A \times I$.
%
%The naturality of $\map^{G}(B,J_Z)$ with respect to $B$, hence that of $\psi$, implies that $\Psi_1 \circ c_{31} = d_{31} \circ \Pi_3$ and $\Psi_0 \circ c_{20} = d_{20} \circ \Psi_2$ and also $\Psi_2 \circ c_%{32}=d_{32} \circ \Psi_3$ and $\Psi_0 \circ c_{10}=d_{10} \circ \Psi_1$.
\hfill QED.

\iffalse
\noindent
{\em Claim 3:} The maps $\Pi_0,\dots,\Pi_3$ in \eqref{E:def Pi B to D} define a morphism $\Pi \colon \B \to \D$ in $\catsq$.
Moreover, $\Lambda(\B,\ul{y}) \xto{\Lambda(\Pi)} \Lambda(\D,\Pi(\ul{y}))$ is a weak homotopy equivalence for any basepoint $\ul{y}$ in $\B$.

\noindent
Proof: By applying $\map^G(-,XYZ)$ to the commutative squares
\[
\xymatrix{
X' \times YZ  \times I \ar[r]^(0.6){\pi} \ar[d]_i &
X'YZ \ar[d]^{i}
\\
X \times YZ \times I \ar[r]^(0.6){\pi} & XYZ
}
\quad
\xymatrix{
(X' \times YZ) \coprod (X' \times YZ) \ar[rr]^(0.6){\pi_{YZ} \coprod \pi_{X'}} \ar@{^(->}[d] & &
YZ \coprod X'  \ar@{^(->}[d]
\\
(X \times YZ) \coprod (X \times YZ) \ar[rr]^(0.6){\pi_{YZ} \coprod \pi_{X}} & &
YZ \coprod X.
}
\]
it follows that $d_{31} \circ \Pi_3 = \Pi_1 \circ b_{31}$ and $d_{20} \circ \Pi_2 = \Pi_0 \circ b_{20}$.
Applying the functor $\map^G(-,XYZ)$ to the pushout square \eqref{E:join pushout} shows that the two squares
\[
\xymatrix{
\B_3 \ar@{->>}[r]^{b_{32}} \ar[d]_{\Pi_3} & \B_2 \ar[d]^{\Pi_2} \\
\D_3 \ar@{->>}[r]^{d_{32}} & \D_2 
}
\qquad
\xymatrix{
\B_1 \ar@{->>}[r]^{b_{10}} \ar[d]_{\Pi_1} & \B_0 \ar[d]^{\Pi_0} \\
\D_1 \ar@{->>}[r]^{d_{10}} & \D_0
}
\]
are pullback squares.
In particular they commute and this shows that the maps $\Pi_i$ define a morphism $\B \xto{\Pi} \D$.
Since the horizontal arrows are fibrations, these are homotopy pullback squares and Lemma \ref{L:elementary facts on Lambda}\ref{L:elementary facts on Lambda 1} shows that $\Lambda(\Pi,\ul{y})$ are weak homotopy equivalences.
\hfill QED.
\fi

\noindent
{\em Claim 3:} The maps $\Pi_0,\dots,\Pi_3$ in \eqref{E:def Pi A to C} define a morphism $\Pi \colon \A \to \C$ in $\catsq$.
Moreover $\Lambda(\A,\ul{x}) \xto{\Lambda(\Pi)} \Lambda(\C,\Pi(\ul{x}))$ is a weak homotopy equivalence for any basepoint $\ul{x}$ for $\A$.
Similarly, the maps $\Pi_0,\dots,\Pi_3$ in \eqref{E:def Pi B to D} define a morphism $\Pi \colon \B \to \D$ in $\catsq$ and $\Lambda(\B,\ul{y}) \xto{\Lambda(\Pi)} \Lambda(\D,\Pi(\ul{y}))$ is a weak homotopy equivalence for any basepoint $\ul{y}$ in $\B$.

\noindent
{\em Proof:} 
We will prove the statements about the maps in \eqref{E:def Pi A to C} and $\Pi \colon \A \to \C$. 
The proof for the maps in \eqref{E:def Pi B to D} and  $\Pi \colon \B \to \D$ is obtained by replacing $Y$ with $YZ$ everywhere and $\A$ with $\B$ and $\C$ with $\D$.

By applying $\map^G(-,XY)$ to the commutative squares
\[
\xymatrix{
X' \times Y \times I \ar[r]^(0.6){\pi} \ar[d]_i &
X'Y \ar[d]^{i}
\\
X \times Y \times I \ar[r]^(0.6){\pi} & XY
}
\quad
\xymatrix{
(X' \times Y) \coprod (X' \times Y) \ar[rr]^(0.6){\pi_{Y} \coprod \pi_{X'}} \ar@{^(->}[d] & &
Y \coprod X'  \ar@{^(->}[d]
\\
(X \times Y) \coprod (X \times Y) \ar[rr]^(0.6){\pi_{Y} \coprod \pi_{X}} & &
Y \coprod X.
}
\]
it follows that $c_{31} \circ \Pi_3 = \Pi_1 \circ a_{31}$ and $c_{20} \circ \Pi_2 = \Pi_0 \circ a_{20}$.
%Since $i_{X'}^X \circ \pi = \pi \circ i_{X'}^X$ \todo{Abuse of notation with $i_{X'}^X$}  where $i \colon X' \to X$ is the inclusion and $\pi \colon A \times Y \to Y$ and $\pi \colon A \times Y \to A$ are the projections, it follows that $d_{20} \circ \Pi_2 = \Pi_0 \circ b_{20}$.
By applying $\map^G(-,XY)$ to the pushout square \eqref{E:join pushout} we obtain pullback squares
\[
\xymatrix{
\A_3 \ar@{->>}[r]^{a_{32}} \ar[d]_{\Pi_3} & \A_2 \ar[d]^{\Pi_2} \\
\C_3 \ar@{->>}[r]^{c_{32}} & \C_2 
}
\qquad
\xymatrix{
\A_1 \ar@{->>}[r]^{a_{10}} \ar[d]_{\Pi_1} & \A_0 \ar[d]^{\Pi_0} \\
\C_1 \ar@{->>}[r]^{c_{10}} & \C_0
}
\]
which are in particular commutative.
Thus, $\Pi_0,\dots,\Pi_3$ define a natural transformation $\Pi \colon \A \to \C$.
Since the horizontal arrows $a_{32}, a_{10}$ and $c_{32}$ and $c_{10}$ are fibrations, the two squares above are homotopy pullback squares and Lemma \ref{L:elementary facts on Lambda}\ref{L:elementary facts on Lambda 1} shows that $\Lambda(\Pi,\ul{x})$ are weak homotopy equivalences for all basepoints $\ul{x}$ of $\A$.
\hfill QED.

Claims 1--3 complete steps (a) and (b) in our plan of the proof.

\noindent
{\em Claim 4:} $\Lambda(\C,\ul{y}) \xto{\Lambda(\Psi)} \Lambda(\D,\Psi(\ul{y}))$ is a $k$-equivalence for any choice of basepoint $\ul{y}$ in $\C$.

\noindent
{\em Proof:}
Hypotheses \ref{P:stabilisation key to WH reduction hyp 0} and \ref{P:stabilisation key to WH reduction hyp 2} %together with the fact that $A \times I$ is $G$-equivalent to $A$,  
allow us to apply Lemma \ref{L:k equivalence on fibres} to the following commutative squares
\[
\xymatrix{
\map^{G}(X \times I,F(Y,XY)) \ar@{->>}[r]^{i^*} \ar[d]_{J_Z^*} &
\map^{G}(X' \times I, F(Y,XY)) \ar[d]^{J_Z^*} \\
\map^{G}(X \times I, F(YZ,XYZ)) \ar@{->>}[r]^{i^*} &
\map^{G}(X' \times I,F(YZ,XYZ))
}
\]
\[
\xymatrix{
\map^{G}(X, F(Y,XY))^2 \ar@{->>}[r]^{i^* \times i^*} \ar[d]_{J_Z^* \times J_Z^*} &
\map^{G}(X',F(Y,XY))^2 \ar[d]^{J_Z^* \times J_Z^*}
\\
\map^{G}(X, F(YZ,XYZ))^2 \ar@{->>}[r]^{i^* \times i^*} &
\map^{G}(X',F(YZ,XYZ))^2 
}
\]
It follows that the fibres of the rows in the $1$st square are $k$-equivalent and those in the $2$nd square are $(k+1)$-equivalent.
%vertical maps in each square induce $(k+1)$-equivalences on the homotopy fibres of the rows.
By construction of the maps $\Psi_0,\dots,\Psi_3$ in \eqref{E:def Psi} and Definition \ref{D:def psiAXYZ} of the maps $\psi$,  it follows that the maps induced on the homotopy fibres in the following diagrams
%By taking adjoints, this is the claim that the map induced on (homotopy) fibres in the following diagrams are $(k+2)$-equivalences
\[
\xymatrix{
\hofib(c_{31},y_1) \ar[r] \ar[d] &
\C_3 \ar[d]_{\Psi_3} \ar@{->>}[r]^{c_{31}} &
\C_1 \ar[d]^{\Psi_1} \\
\hofib(d_{31},\Psi_1(y_1)) \ar[r]  &
\D_3  \ar@{->>}[r]^{d_{31}} &
\D_1
}
\qquad
\xymatrix{
\hofib(c_{20},y_0) \ar[r] \ar[d] &
\C_2 \ar[d]_{\Psi_2} \ar@{->>}[r]^{c_{20}} &
\C_0 \ar[d]^{\Psi_0} \\
\hofib(d_{20},\Psi_0(y_0)) \ar[r]  &
\D_2  \ar@{->>}[r]^{d_{20}} &
\D_0
}
\]
are $(k+1)$-equivalences for any choice of basepoints $y_0 \in \C_0$ and $y_1 \in \C_1$.
Lemma \ref{L:elementary facts on Lambda}\ref{L:elementary facts on Lambda 3} shows that $\Lambda(\C,\ul{y}) \xto{\Lambda(\Psi)} \Lambda(\D,\Psi(\ul{y}))$ is a $k$-equivalence for any basepoint $\ul{y}$ of $\C$.
% Apply Lemma \ref{L:elementary facts on Lambda}\ref{L:elementary facts on Lambda 3} to deduce that $\Lambda(\C,\ul{y}) \xto{\Lambda(\Psi)} \Lambda(\D,\Psi(\ul{y}))$ is a $(k+1)$-equivalence.
\hfill QED.

This completes step (c) of the proof.
We turn to the technical proof of step (d).

\noindent
{\em Claim 5:} The morphisms $\A \xto{\Phi} \B \xto{\Pi} \D$ and $\A \xto{\Pi} \C \xto{\Psi} \D$ are homotopic (Definition \ref{D:homotopy in catsq}).

\noindent
{\em Proof:}
The plan is define an object $\H \in \catsq$, a morphism $\A \xto{\Upsilon} \H$ and a homotopy $\H \xto{\Xi_p} \D$ parameterized by $0 \leq p \leq 1$, such that $(\Xi \circ \Upsilon)|_{p=0} = \Pi \circ \Phi$ and $(\Xi \circ \Upsilon)|_{p=1} = \Psi \circ \Pi$.

By applying the functor $\map^G(-,XYZ)$ to the commutative square of inclusions
\[
\xymatrix{
YZ \coprod X'Z \ar@{^(->}[r] \ar@{^(->}[d] & X'YZ \ar@{^(->}[d] \\
YZ \coprod XZ \ar@{^(->}[r] & XYZ
}
\]
we obtain the following object $\H$ in $\catsq$.
\[
\H = 
\vcenter{
\xymatrix{
\map^G(XYZ,XYZ) \ar[rr]^(0.4){(i_{YZ}^*,i_{XZ}^*)} \ar@{->>}[d]_{i^*} & &
\map^G(YZ,XYZ) \times \map^G(XZ,XYZ) \ar@{->>}[d]^{\id \times i^*}
\\
\map^G(X'YZ,XYZ) \ar[rr]^(0.4){(i_{YZ}^*,i_{X'Z}^*)} & &
\map^G(YZ,XYZ) \times \map^G(X'Z,XYZ).
}
}
\]
% The square commutes because $i_{XZ}^{XYZ} \circ i_{X'Z}^{XZ} = i_{X'YZ}^{XYZ} \circ i_{X'Z}^{X'YZ}$ and we get an object $\H \in \catsq$.
Define maps $\Upsilon_i \colon \A_i \to \H_i$ ($i=0,\dots,3$) as follows (we use the indexing as in Definition \ref{D:cat sq}).
\begin{eqnarray}
\label{E:def Upsilon}
&& \Upsilon_1, \Upsilon_3 \colon \map^G(AY,XY) \xto{J_Z} \map^G(AYZ,XYZ) \\
\nonumber
&& \Upsilon_0,\Upsilon_2 \colon \map^G(Y,XY) \times \map^G(A,XY) \xto{J_Z \times J_Z} \map^G(YZ,XYZ) \times \map^G(AZ,XYZ).
\end{eqnarray}
They give rise to a morphism $\Upsilon \colon \A \to \H$ in $\catsq$ by commutativity of the square in \eqref{E:JZ incl}.

We now define homotopies $(\Xi_i)_p \colon \H_i \to \D_i$ parameterized by $0 \leq p \leq 1$. 
Apply $\map^G(-,XYZ)$ to the homotopies $\Theta_p \colon A \times YZ \times I \to AYZ$ in Definition \ref{D:Thetap homotopy} to obtain the maps
\begin{equation}\label{E:def Xi-13}
(\Xi_3)_p \colon  \H_3 \xto{(\Theta_p)^*} \D_3 \qquad \text{  and  }  \qquad (\Xi_1)_p \colon  \H_1 \xto{(\Theta_p)^*} \D_1.
\end{equation}
Definition \ref{D:Thetap homotopy} and Proposition \ref{P:tilde theta} show that $\Theta_p|_{t=0} \defeq \Theta_p|_{A \times YZ \times \{0\}}$ factors through the inclusion $YZ \subseteq AYZ$ and that  $\Theta_p|_{t=1} \defeq \Theta_p|_{A \times YZ \times \{1\}}$ factors through the inclusion $AZ \subseteq AYZ$.
By applying $\map^G(-,XYZ)$ to these square homotopies we obtain the maps
\begin{equation}\label{E:def Xi-02}
(\Xi_2)_p \colon \H_2 \xto{(\Theta_p|_{t=0})^* \times (\Theta_p|_{t=1})^*} \D_2 \qquad \text{and} \qquad
(\Xi_0)_p \colon \H_2 \xto{(\Theta_p|_{t=0})^* \times (\Theta_p|_{t=1})^*}  \D_0.
\end{equation}
Naturality of the construction of $\Theta$ with respect to the inclusion $X' \subseteq X$ implies the commutativity of the squares
\[
\xymatrix{
\H_3 \ar[r]^{h_{31}} \ar[d]_{(\Xi_3)_p} &
\H_1 \ar[d]^{(\Xi_1)_p} 
\\
\D_3 \ar[r]^{d_{31}} & \D_1
}
\qquad
\xymatrix{
\H_2 \ar[r]^{h_{20}} \ar[d]_{(\Xi_2)_p} &
\H_0 \ar[d]^{(\Xi_0)_p} 
\\
\D_2 \ar[r]^{d_{20}} & \D_0.
}
\]
Furthermore, by applying $\map^G(-,XYZ)$ to the commutative square
\begin{equation}\label{E:Xi 0-2 commutativity key}
\vcenter{
\xymatrix{
(A \times YZ) \sqcup (A \times YZ) \ar[rr]^(0.6){\incl_0 + \incl_1} \ar[d]_{\Theta_p|_{t=0} \sqcup \Theta_p|_{t=1}} & &
A \times YZ \times I \ar[d]^{\Theta_p}
\\
YZ \sqcup AZ \ar[rr]_{i_{YZ} + i_{AZ}} & &
AYZ.
}
}
\end{equation}
we obtain the commutativity of 
\[
\xymatrix{
\H_3 \ar[r]^{h_{32}} \ar[d]_{(\Xi_3)_p} &
\H_2 \ar[d]^{(\Xi_2)_p} 
\\
\D_3 \ar[r]^{d_{32}} & \D_2
}
\qquad
\xymatrix{
\H_1 \ar[r]^{h_{10}} \ar[d]_{(\Xi_1)_p} &
\H_0 \ar[d]^{(\Xi_0)_p} 
\\
\D_1 \ar[r]^{d_{10}} & \D_0
}
\]
Therefore $\Xi_0,\dots,\Xi_3$ give rise to a homotopy $\Xi \colon I \times \H \to \D$ in $\catsq$.
Composition with $\Upsilon \colon \A \to \H$ gives a homotopy $\Xi \circ \Upsilon \colon I \times \A \to \D$ parameterized by $p \in I$.

It remains to show that $(\Xi \circ \Upsilon)|_{p=0} =\Psi \circ \Pi$ and that $(\Xi \circ \Upsilon)|_{p=1} =\Pi \circ \Phi$.
We start with $p=1$.
By Proposition \ref{P:def alpha and beta via Theta}, $\Theta|_{p=1}=\beta$  is the natural map $A \times YZ  \times I \xto{\pi} A(YZ) \cong AYZ$ in \eqref{E:join pushout}. % described in \eqref{E:def alpha and beta via Theta}.
By the definition of $\Phi_1, \Phi_3$ in \eqref{E:def Phi} and $\Pi_1, \Pi_3$ in \eqref{E:def Pi B to D} and $\Upsilon_1, \Upsilon_3$ in \eqref{E:def Upsilon} and $\Xi_1, \Xi_3$ in \eqref{E:def Xi-13}, it follows that for $i=1,3$
\[
(\Xi_i)|_{p=1} \circ \Upsilon_i = \beta^* \circ J_Z = \pi^* \circ J_Z = \Pi_i \circ \Phi_i.
\]
Proposition \ref{P:def alpha and beta via Theta} shows that $\Theta_{p=1}|_{t=0}=\beta|_{A \times YZ \times \{0\}}$ is the projection $A \times YZ \to YZ$ and that $\Theta_{p=1}|_{t=1}=\beta|_{A \times YZ \times \{1\}}$ is the composition of the projection $A \times YZ \to A$ followed by the inclusion into $AZ$.
Since $(\Xi_0)_p$ and $(\Xi_2)_p$ are obtained by applying $\map^G(-,XYZ)$ to the first column of \eqref{E:Xi 0-2 commutativity key}, the commutative triangle in \eqref{E:JZ incl} together with \eqref{E:def Phi} and \eqref{E:def Pi B to D} show that for $i=0,2$ 
\[
%\begin{multline*}
(\Xi_i)_{p=1} \circ \Upsilon_i = 
(\pi_{YZ}^* \circ J_Z) \times (\pi_A^* \circ (i_A^{AZ})^* \circ J_Z) = 
%\\
(\pi_{YZ}^* \circ J_Z) \times (\pi_A^* \circ \incl_*)  = \Pi_i \circ \Phi_i.
%\end{multline*}
\]
It follows that
\[
(\Xi \circ \Upsilon)|_{p=1} = \Pi \circ \Phi.
\]
It remains to show that $(\Xi \circ \Upsilon)|_{p=0} = \Psi \circ \Pi$.
First, we claim that for $i=1,3$
\[
(\Xi_i)_{p=0} \circ \Upsilon_i = \alpha^* \circ J_Z= \psi_{A\times I,Y,XY,Z} \circ \pi^* =\Psi_i \circ \Pi_i.
\]
The first equality follows from the definitions of $\Xi_i$ and $\Upsilon_i$ in \eqref{E:def Xi-13} and \eqref{E:def Upsilon}, where $i=1,3$, and from Definition \ref{D:alpha and beta};
The second equality follows from the commutative square \eqref{E:JZ alpha} in Proposition \ref{P:more on alpha beta restricted}, and the third from the definitions of $\Psi_i$ and $\Pi_i$ in \eqref{E:def Psi} and \eqref{E:def Pi A to C}.

Let $i=0,2$.
We claim that 
\begin{multline*}
(\Xi_i)_{p=0} \circ \Upsilon_i = (\alpha|_{A \times YZ \times\{0\}}{}^* \times \alpha|_{A \times YZ \times\{1\}}{}^*) \circ (J_Z \times J_Z) = % (\alpha_0^* \circ J_Z) \times (\alpha_1^* \circ J_Z) =
\\
(\psi_{A,Y,XY,Z} \circ (\pi_{A\times Y}^Y)^*) \times (\psi_{A,Y,XY,Z} \circ (\pi_{A\times Y}^A)^*) = \Psi_i \circ \Pi_i.
\end{multline*}
The first equality follows from the definition of $\Upsilon_i$ and $\Xi_i$ in \eqref{E:def Upsilon} and \eqref{E:def Xi-13};
The second follows from \eqref {E:JZ alpha0} and \eqref{E:JZ alpha1} in Proposition \ref{P:more on alpha beta restricted}, and the third from the definition of $\Psi_i$ and $\Pi_i$ in \eqref{E:def Psi} and \eqref{E:def Pi A to C}.
It follows that $(\Xi \circ \Upsilon)|_{p=0} = \Psi \circ \Pi$.
\hfill Q.E.D

This completes step (d) of the proof. 
We are now ready to complete the proof of the proposition as outlined in step (e).

Claims 3 and 4 and the functoriality of $\Lambda$ imply that $\Lambda(\A,\ul{x}) \xto{\Lambda(\Psi \circ \Pi)} \Lambda(\D,\Psi\circ \Pi(\ul{x}))$ is a $k$-equivalence for any choice of basepoint $\ul{x}$ in $\A$.
Claim 5 together with Lemma \ref{L:homotopic morphisms in catsq} show that $\Lambda(\A,\ul{x}) \xto{\Lambda(\Pi \circ \Phi)} \Lambda(\D,\Pi\Phi(\ul{x}))$ is a $k$-equivalence.
From Claim 3 and the functoriality of $\Lambda$ we deduce that $\Lambda(\A,\ul{x}) \xto{\Lambda(\Phi)} \Lambda(\B,\Phi(\ul{x}))$ is a $k$-equivalence for any basepoint $\ul{x}$ in $\A$.

By hypothesis \ref{P:stabilisation key to WH reduction hyp 4}, $\conn (XY) \geq \dim X + k+1$ and therefore also $\conn(XYZ) \geq \dim X + k+1$, see for example \cite[Lemma 2.3]{Mil56II}.
Thanks to hypothesis \ref{P:stabilisation key to WH reduction hyp 0}, we may apply  Corollary \ref{C:maps with highly connected target Y} to $X' \subseteq X$ and to $XY$ and $XYZ$ and deduce, in light of the definitions of $\A$ and $\B$, that the horizontal arrows in the following square are $(k+1)$-equivalences.
\[
\xymatrix{
\A_2 \ar@{->>}[r]^{a_{20}} \ar[d]_{\Phi_2} &
\A_0 \ar[d]^{\Phi_0} 
\\
\B_2 \ar@{->>}[r]^{b_{20}} & \B_0
}
\]
In particular, their fibres are $(k+1)$-connected, and therefore $\Phi_2$ and $\Phi_0$ induce $k$-equivalences among them.
We have already seen that $\Lambda(\A,\ul{x}) \xto{\Lambda(\Phi)} \Lambda(\B,\Phi(\ul{x}))$ are $k$-equivalences, therefore we may apply Lemma \ref{L:elementary facts on Lambda}\ref{L:elementary facts on Lambda 4} to deduce that in the commutative square
\[
\xymatrix{
\A_3 \ar@{->>}[r]^{a_{31}} \ar[d]_{\Phi_3} & 
\A_1  \ar[d]^{\Phi_1} 
\\
\B_3 \ar@{->>}[r]^{b_{31}} &
\B_1
}
\]
the vertical arrows induce $k$-equivalences on all the fibres of $a_{31}$ and $b_{31}$.
Given the definition of $\Phi_1$ and $\Phi_3$ in \eqref{E:def Phi}, this is exactly the claim of this proposition.
\end{proof}

%%%
%%% --------
%%% SECTION: The stable term
%%%
\section{The limit groups}

The purpose of this section is to prove Proposition \ref{P:stable term X} below.
It will make an essential use of equivariant stable homotopy theory.
Our main reference is Lewis-May-Steinberg \cite{LMS86}.
Let $G$ be a finite group.
We will fix a universe $\U$ for the representations of $G$ and an indexing set $\A$ for $\U$ \cite[Definition 2.1]{LMS86}.
A spectrum $E$ is a collection of pointed $G$-spaces $E(V)$ for every $V \subseteq \A$, subject to some conditions \cite[Def. 2.1]{LMS86}.
There is a functor $\Sigma^\infty$ from the category of pointed $G$-spaces to the category of $G$-spectra.
It has a right adjoint $\Omega^\infty$ which assigns to a spectrum $E$ its zeroth space \cite[Proposition 2.3]{LMS86}.
We denote, as usual, $Q=\Omega^\infty\circ \Sigma^\infty$ \cite[p. 14]{LMS86}.
If $A$ is a pointed $G$-space, the function spectrum $F(A,E)$ is obtained by applying $F(A,-)$ to each space of $E$ \cite[Def. 3.2]{LMS86}.
If $H \leq G$, there is a fixed point spectrum $E^H$ which is obtained, once again, by taking $H$-fixed points of the spaces of $E$ \cite[Def. 3.7]{LMS86}.
It has a structure of a $WH$-spectrum where $WH=N_GH/H$.

The Borel construction of a $G$-space $X$ is the orbit space $EG \times_G X = (EG \times X)/G$.
Let $X$ be a pointed $G$-space.
Denote $EG_+ \wedge_G X \defeq (EG_+ \wedge X)/G$.
The following important result \cite[Section V.11]{LMS86} gives a complete description of the fixed point spectrum $(\Sigma^\infty X)^G$ 
\[
(\Sigma^\infty X)^G \simeq \bigvee_{H}  \Sigma^\infty \left( EWH_+ \underset{WH}{\wedge} X^H\right)
\]
where $H$ runs through representatives of the conjugacy classes of the subgroups of $G$ and where $WH=N_GH/H$ acts in the natural way on $X^H$.

Let $U$ be a real representation of $G$.
Let $S^U$ denote the one point compactification of $U$ with basepoint $\alpha_1=\infty$.
Clearly $\alpha_0 =0 \in U$ is also fixed by $G$.
Notice that for any $H \leq G$ we have $(S^U)^H = S^{V}$ where $V=U^H$.
% It has a natural action of $WH$.

Let $C(f)$ denote the mapping cone of a map $f \colon A \to B$ of unpointed $G$-spaces; it is equipped with a natural basepoint (the ``tip of the cone'').
If $A \subseteq B$ we write $C(B,A)$ for the mapping cone of the inclusion.

For a $G$-space $X$ with fixed points $x_0,x_1$, denote by $\PP_{x_0,x_1}X$ the space of paths $\omega \colon I \to X$ with $\omega(0)=x_0$ and $\omega(1)=x_1$.
It has a natural action of $G$.

\begin{lemma}\label{L:vanishing homology and stable homotopy}
\begin{enumerate}
\item
\label{I:vanishing homology:1}
Let $X$ be a pointed finite CW-complex such that $\tilde{H}_i(X)=0$
for all $0 \leq i \leq m$.  Then $\pi_i \Sigma^\infty X =0$ for all $0 \leq i \leq m$.

\item
\label{I:vanishing homology:2}
Let $X$ be a finite $G$-CW complex such that $\tilde{H}_i(X)=0$ for all $0 \leq i \leq m$.
Then $\pi_i\Sigma^\infty (EG_+ \wedge_G X)=0$ for all $0 \leq i \leq m$.
\end{enumerate}
\end{lemma}

\begin{proof}
\ref{I:vanishing homology:1}.
This follows from Atiyah-Hirzebruch spectral sequence $\tilde{H}_i(X,E_j(*)) \Rightarrow \tilde{E}_{i+j}(X)$ applied to the sphere spectrum $E=\mathbf{S}$.

\ref{I:vanishing homology:2}.
There is a $G$-cofibre sequence where $EG$ retracts off $EG \times X$ equivariantly (via the basepoint of $X$).
\[
\xymatrix{
EG \ar[r] &
EG \times X \ar[r] \ar@/_/ @{-->}[l] &
EG_+ \wedge X
}
\]
By taking $G$-orbits we get a cofibre sequence
\[
\xymatrix{
BG \ar[r] &
X_{hG} \ar[r] \ar@/_/ @{-->}[l] &
EG_+ \wedge_G X
}
\]
with $BG$ retracting off $X_{hG}=EG \times_G X$.
The Serre spectral sequence $H_i(BG,H_j(X)) \Rightarrow H_{i+j}(X_{hG})$ of the fibration $X_{hG} \to BG$ shows that $H_i(BG) \to H_i(X_{hG})$ is an isomorphism for all $0 \leq i \leq m$ and therefore $\tilde{H}_i(EG_+ \wedge_G X)=0$ for all $0 \leq i \leq m$.
The result follows from item \ref{I:vanishing homology:1}.
\end{proof}

\begin{lemma}\label{L:mapping cone dimension shift}
Let $V$ be a $G$-representation and $X \subseteq V$ a finite $G$-CW complex.
Set $n=\dim V$.
Then there are isomorphisms for all $0 \leq i \leq n-2$
\[
\pi_{i+1} \Sigma^\infty \left( EG_+ \wedge_G C(S^V,S^V \setminus X)\right) \cong
\pi_i \Sigma^\infty (EG_+ \wedge_G (S^V \setminus X))
\]
\end{lemma}
\begin{proof}
Lemma \ref{L:vanishing homology and stable homotopy}\ref{I:vanishing homology:2} shows that $\pi_i \Sigma^\infty(EG_+ \wedge_G S^V)=0$ for $0 \leq i \leq n-1$.
The long exact sequence in stable homotopy groups of the cofibration $EG_+\wedge_G (S^V \setminus X) \to EG_+ \wedge_G S^V \to EG_+ \wedge_G C(S^V,S^V \setminus X)$ gives the result.
\end{proof}

\begin{prop}\label{P:stable term X}
Let $U$ be a representation of $G$ and let $X \subseteq U$ be a finite $G$-CW complex.
Let $k \geq 0$.
Assume that for any $H \in \Iso_G(X)$
\begin{enumerate}[label=(\alph*)]%[label=(\roman*)]
\item 
$\dim X^H < \dim U^H$.
\label{P:stable term X:1}

\item
$\dim X^H - \dim \left(\bigcup_{K \gneq H}X^K\right) \ > \ k+1$.
\label{P:stable term X:2}

\item
$(S^U \setminus X)^H$ is $WH$-equivariantly homotopy equivalent to a $WH$-CW complex.
\label{P:stable term X:3}
\end{enumerate}
Then $\map^G(X,S^U)$ is path connected and for all $1 \leq i \leq k+1$ % (and any choice of basepoint)
\begin{equation}\label{E:stable term X:1}
\pi_i \map^G(X,S^U) \cong \bigoplus_{(H) \subseteq \Iso_G(X)} \pi_i \Sigma^{\infty}\left( EWH_+ \underset{WH}{\wedge} C(S^U,S^U \setminus X)^H \right). %\todo{$C(-,-)$ and $h_\bullet WH$}
\end{equation}
If in addition 
\begin{enumerate}[label=(\alph*)]%[label=(\roman*)]
\setcounter{enumi}{3}
\item
there exists a $G$-map $X \xto{\eta} \PP_{\alpha_0,\alpha_1}S^U$ such that $X^H \xto{\eta} \PP_{\alpha_0,\alpha_1}(S^U)^H$ is a $(\dim X^H+k)$-equivalence for any $H \in \Iso_G(X)$
\label{P:stable term X:4}
\end{enumerate}
then for any basepoint $f \in \map^G(X,X)$ and every $0 \leq i \leq k$ there are isomorphisms (bijection for $i=0$):
\begin{equation}\label{E:stable term X 2}
\pi_i \map^G(X,X) \cong \bigoplus_{(H) \subseteq \Iso_G(X)} \pi_i \Sigma^\infty\left(EWH_+ \underset{WH}{\wedge} (S^U \setminus X)^H\right).
\end{equation}
\end{prop}

\begin{proof}%[Proof of Proposition \ref{P:stable term X}]
We will prove by induction on the filtration $\{ X_q\}_{q=0}^r$ of $X$ in \eqref{E:def filtration X_q} that $\map^G(X_q,S^U)$ is path connected and that there are isomorphisms for all $1 \leq i \leq k+1$
\begin{equation}\label{E:stable term X ind step}
\pi_i \map^G(X_q,S^U) \cong \bigoplus_{(H) \subseteq \Iso_G(X_q)} \pi_i \Sigma^\infty \left(EWH_+ \underset{WH}{\wedge} C(S^U , S^U \setminus X)^H\right).
\end{equation}
The base of induction is a triviality since $X_0=\emptyset$.
Assume that \eqref{E:stable term X ind step} holds for $q-1$ and we prove it for $1 \leq q \leq r$.
If $H_q \notin \Iso_G(X)$ then $X_q=X_{q-1}$ and the induction step is trivial.
So we assume that $H=H_q$ is in $\Iso_G(X)$.

Choose some basepoint $f \in \map^G(X_q,S^U)$.
We obtain a fibre sequence (over $f|_{X_{q-1}}$)
\[
F \to \map^G(X_q,S^U) \xto{j^*} \map^G(X_{q-1},S^U).
\]
The hypotheses imply that 
\[
\dim_H (X_q,X_{q-1}) \leq \dim X_q^H=\dim X^H  \leq \dim U^H-1 = \conn (S^U)^H.
\]
We can apply Corollary \ref{C:maps with highly connected target Y} (with $Y=S^U$ and $k=0$) to deduce that $j^*$ is bijective on components and that $\pi_0 F=*$. 
Together with the induction hypothesis on $\map^G(X_{q-1},S^U)$, it follows that $\map^G(X_q,S^U)$ is path connected, as needed.
Therefore we may assume that the basepoint $f$ is the null map.

Since $\Iso_G(X_q) = \Iso_G(X_{q-1}) \cup (H)$, in order to complete the induction step for \eqref{E:stable term X ind step} it remains to show that for every $1 \leq i \leq k+1$

(i) $\pi_i F \to \pi_i \map^G(X_q,S^U)$ is split injective , and

(ii) $\pi_iF \cong \pi_i \Sigma^\infty \left(EWH_+ \wedge_{WH} C(S^U,S^U \setminus X)^H\right)$.

For the rest of the proof set $V=U^H$ and $n=\dim V$.
Proposition \ref{P:isom on fibres in F_q filration} yields the following morphism of fibrations which induces a homeomorphism on the fibres (over the null maps)
\[
\xymatrix{
F \ar[r] \ar[d]_{\cong} &
\map^G(X_q,S^U) \ar@{->>}[r]^{j^*} \ar[d] &
\map^G(X_{q-1},S^U) \ar[d] 
\\
F \ar[r] &
\map^{WH}(X^H,S^V) \ar@{->>}[r]^{j^*} &
\map^{WH}(X_{q-1}^H,S^V).
}
\]
Therefore, we will be finished if we prove (ii) and that $\pi_i F \to \pi_i\map^{WH}(X_{q-1}^H,S^V)$ is split injective for all $1 \leq i \leq k+1$.

Since $QS^V=\colim_{V'} \Omega^{V'}\Sigma^{V'} S^V$ \cite[p. 14]{LMS86}, Freudenthal's theorem implies that the natural map $S^V \to QS^V$ is a non-equivariant $(2n-2)$-equivalence.
Hypotheses \ref{P:stable term X:1} and \ref{P:stable term X:2} imply that 
\[
\dim X^H \geq k+2 + \dim \left(\bigcup_{K \gneq H}X^K\right) \geq  k+1
\]
because $\dim \, (\cup_{K \gneq H} X^K) \geq -1$.
In particular
\[
\dim (X^H,X_{q-1}^H) + k+1 \leq \dim X^H + k+1 \leq 2\cdot \dim X^H  \leq 2(n-1) = 2n-2.
\]
Application of Proposition \ref{P:JZ and WH}\ref{P:JZ and WH:free action of WH on X-X'} and Lemma \ref{L:k equivalence on fibres} to $X_{q-1}^H \subseteq X^H$ and $S^V \to QS^V$ shows that in the commutative diagram
\begin{equation}\label{E:stable term X:QSV}
\xymatrix{
F \ar[r] \ar[d] &
\map^{WH}(X^H,S^V) \ar[d] \ar@{->>}[r]^{j^*} &
\map^{WH}(X_{q-1}^H,S^V) \ar[d] 
\\
F' \ar[r] &
\map^{WH}(X^H,QS^V) \ar@{->>}[r]^{\ell^*} &
\map^{WH}(X_{q-1}^H,QS^V).
}
\end{equation}
the map $F \to F'$ between the fibres (over the null maps) is a $(k+1)$-equivalence.
Thus, to complete the induction step of \eqref{E:stable term X ind step} it suffices to prove that for every $1 \leq i \leq k+1$

(i') $\pi_i F' \to \pi_i\map^{WH}(X^H,QS^V)$ is split injective, and 

(ii') $\pi_iF'$ is isomorphic to the groups in (ii).

Let $A$ be a pointed $G$-CW complex.
The definitions of $\Omega^\infty$ and $Q=\Omega^\infty \Sigma^\infty$ \cite[p. 14]{LMS86}, of the function spectra $F(A,E)$ \cite[Prop. 3.6]{LMS86}, and of fixed point spectra \cite[Def. 3.7]{LMS86}, imply that there are natural homeomorphisms
\begin{multline*}
\map^{WH}(A,QS^V) = F(A_+,\Omega^\infty \Sigma^\infty S^V)^{WH}
\cong \left(\Omega^\infty F(A_+,\Sigma^\infty S^V)\right)^{WH} \\
= \Omega^\infty \left(F(A_+,\Sigma^\infty S^V)^{WH}\right).
\end{multline*}
Therefore the fibration $\ell^*$ in the 2nd row of \eqref{E:stable term X:QSV}  is obtained by applying the functor $\Omega^\infty$ and $WH$-fixed points to the morphism of $WH$-spectra
\[
F(X^H_+,\Sigma^\infty S^V) \xto{\ell^*} F((X_{q-1}^H)_+,\Sigma^\infty S^V).
\]
We will now exploit $V$-duality \cite[Chap. III]{LMS86}.
Recall that $C(X,\emptyset)=X_+$ where $X$ is an unpointed space \cite[page 142]{LMS86}.
The definition of $V$-duality \cite[Defn. 3.4]{LMS86} together with the formula for the map $\epsilon/(?)$ \cite[Prop. 3.1]{LMS86} and the construction of $V$-duality for compact $G$-ENRs \cite[Construction 4.5, page 145]{LMS86}, give rise to the following homotopy commutative diagram of $WH$-spectra
\[
\xymatrix{
F(X^H_+,\Sigma^\infty S^V) \ar[r]^{\ell^*} &
F((X_{q-1})^H_+,\Sigma^\infty S^V) 
\\
\Sigma^\infty C(S^V,S^V \setminus X^H) \ar[r]^{\ell^*} \ar[u]^{\epsilon^\#(X^H)} &
\Sigma^\infty C(S^V,S^V \setminus X_{q-1}^H)  \ar[u]_{\epsilon^\#(X_{q-1}^H)}.
}
\]
By definition of $V$-duality (or by construction), the vertical arrows in this diagram are weak equivalences of $WH$-spectra \cite[Def. 4.4]{LMS86}.
By applying the fixed points functor $(-)^{WH}$, and the functor $\Omega^\infty$ we see that in order to prove (i') and (ii') it suffices to prove that 

(i'')  $\pi_i \left(\Sigma^\infty C(S^V,S^V \setminus X^H)\right)^{WH} \xto{\pi_i(\ell^*)} \pi_i \left(\Sigma^\infty C(S^V,S^V \setminus X_{q-1}^H)\right)^{WH}$ is split surjective for all $1 \leq i \leq k+1$ and surjective for $i=k+2$, and

(ii'') the kernels of the homomorphisms in (i'') are isomorphic to the groups in (ii) for all $1 \leq i \leq k+1$.

On the level of spectra, \cite[Section V.11]{LMS86} quoted above shows that the map in (i'') is induced by the map of spectra
\[
\bigvee_{(K)} \Sigma^\infty (EWK_+ \underset{WK}{\wedge} C(S^V,S^V \setminus X^H)^K) \to 
\bigvee_{(K)} \Sigma^\infty (EWK_+ \underset{WK}{\wedge} C(S^V,S^V \setminus X_{q-1}^H)^K)
\]
where the sum is over all the conjugacy classes of subgroups $K \leq WH$ and by $WK$ we mean $N_{WH}(K)/K$.
Any $K \leq WH$ has the form $K=L/H$ for some $H \leq L \leq N_GH$.
If $K \neq 1$ then $L \supsetneq H$ and in this case $X^L=X_{q-1}^L$ and it follows that the maps of the summands corresponding to $K \neq 1$ are equivalences.
It remains to examine the summand $K=1$, namely the map
\begin{equation}\label{E:tom Dieck reduction to K=1}
\Sigma^\infty( EWH_+ \wedge_{WH} C(S^V, S^V \setminus X^H)) \to \Sigma^\infty (EWH_+ \wedge_{WH} C(S^V, S^V \setminus X_{q-1}^H)).
\end{equation}
The hypotheses and Proposition \ref{P:filtration facts G}\ref{P:filtration facts G:fix H} show that 
\[
\dim X_{q-1}^H  \leq \dim X^H -k-2 \leq n-k-3.
\]
In particular $H^i(X_{q-1}^H)=0$ for all $i \geq n-k-2$.
Alexander duality implies that $\tilde{H}_i(S^V \setminus X_{q-1}^H)=0$ for all $0 \leq i \leq k+1$.
Also, $\tilde{H}_i(S^V)=0$ for all $0 \leq i \leq k+1$ since $\conn(S^V)=n-1 \geq \dim X^H \geq k+1$.
There is a cofibre sequence
%By applying the functor $EWH_+ \wedge_{WH} (-)$ to the cofibre sequence
\[
\xymatrix@1{
EWH_+ \wedge_{WH} (S^V \setminus X_{q-1}^H) \ar@{^(->}[r] & EWH_+ \wedge_{WH} S^V \ar[r]^(0.37)\gamma & EWH_+ \wedge_{WH} C(S^V,S^V \setminus X_{q-1}^H).
}
\]
The long exact sequence in stable homotopy groups together with Lemma \ref{L:vanishing homology and stable homotopy}\ref{I:vanishing homology:2} show that $\pi_i$ of the right hand side of \eqref{E:tom Dieck reduction to K=1} vanishes for $0 \leq i \leq k+1$ and that $\pi_{k+2} \Sigma^\infty \gamma$ is surjective.
Also $\pi_{k+2}$ of \eqref{E:tom Dieck reduction to K=1} is surjective because $\pi_{k+2}\Sigma^\infty\gamma$ factors through it.
\iffalse
We now consider the morphism of cofibre sequences of $WH$-spaces
\[
\xymatrix{
S^V \setminus X^H \ar@{^(->}[d] \ar[r] &
S^V \ar[d]^{\id} \ar[r] &
C(S^V,S^V \setminus X^H) \ar@{^(->}[d] 
\\
S^V \setminus X_{q-1}^H \ar[r] &
S^V \ar[r] &
C(S^V,S^V \setminus X_{q-1}^H). 
}
\]
By applying the functor $EWH_+ \wedge_{WH} -$ we obtain a ladder of cofibre sequences.
A diagram chase of the resulting morphism of the long exact sequences in stable homotopy groups, together with Lemma \ref{L:vanishing homology and stable homotopy}\ref{I:vanishing homology:2} easily show that $\pi_i\Sigma^\infty EWH_+ \wedge_{WH} C(S^V, S^V \setminus X_{q-1}^H)=0$  for all $0 \leq i \leq k+1$ and that the map in \eqref{E:tom Dieck reduction to K=1} induces a surjection of $\pi_{k+2}$.
\fi
In particular (i'') and (ii'') follow and the induction step is complete.

Let $X \xto{\eta} \PP_{\alpha_0,\alpha_1}S^U$ be as in hypothesis \ref{P:stable term X:4}.
Applying Lemma \ref{L:k equivalence on fibres} with $\emptyset \subseteq X$ and with $\eta$ shows that $\map^G(X,X) \to \map^G(X,\PP_{\alpha_0,\alpha_1}S^U)$ is a  $k$-equivalence.
By inspection, and since we have shown that $\map^G(X,S^U)$ is path connected,
\[
\map^G(X,\PP_{\alpha_0,\alpha_1}S^U) \cong \PP_{\alpha_0,\alpha_1}\map^G(X,S^U) \simeq \Omega \map^G(X,S^U).
\]
We have seen that if $H \in \Iso_G(X)$ and $n=\dim U^H$ then $n-1 \geq \dim X^H \geq k+1$.
Lemmas \ref{L:mapping cone dimension shift} and \ref{L:vanishing homology and stable homotopy}\ref{I:vanishing homology:2} apply to show that for $0 \leq i \leq k$ there are isomorphisms (bijection if $i=0$)
\begin{multline*}
\pi_i \map^G(X,X) \cong
\pi_{i+1} \map^G(X,S^U) \cong
\bigoplus_{(H) \subseteq \Iso_G(X)} \pi_{i+1} \Sigma^\infty \left( EWH_+ \wedge_{WH} C(S^U,S^U \setminus X)^H\right) 
\\
\cong
\bigoplus_{(H) \subseteq \Iso_G(X)} \pi_{i} \Sigma^\infty \left( EWH_+ \wedge_{WH} (S^U \setminus X)^H\right) 
\end{multline*}
\end{proof}

%------------------------------------------
% SECTION : Proofs
%------------------------------------------
\section{Proof of the main theorems}

In this section we fix a finite group $G$ and a sequence of representation $U_1, U_2, \dots$ satisfying hypothesis \ref{H:U} in Section \ref{S:Intro}.
For any $m \leq n$ we write 
\[
U_{m\leq \bullet\leq n} = \oplus_{i=m}^n U_i.
\]
Recall that $\F(U_\bullet)$ is the smallest collection of subgroups of $G$ which contains $\Iso_G(S(V))$ for all $V \in \Irr(U_\bullet)$ and is closed to intersection of groups.
Illman showed in \cite{Illman78} that $S(V)$ is a finite $G$-CW complex for any representation $V$.
Also $S(V)^H=S(V^H)$ is a sphere.

\begin{lemma}\label{L:far in U}
Fix $m \geq 1$ and $k \geq 0$.
% Assume hypothesis \ref{H:U}.
Then for every sufficiently large $n$
\begin{enumerate}
\item
\label{L:far in U:1}
$\Iso_G(S(U_{m \leq \bullet \leq n}))=\F(U_\bullet)$. % for all sufficiently large $n$.

\item
\label{L:far in U:2}
$\dim S(U_{m \leq \bullet \leq n})^H \geq k$  for every $H \in \F(U_\bullet)$.

\item
\label{L:far in U:3}
$\dim S(U_{m \leq \bullet \leq n})^H - \dim \bigcup_{K \gneq H} S(U_{m \leq \bullet \leq n})^K \, \geq \, k$ for every $H \in \F(U_\bullet)$.
(In this inequality the dimension of the empty set is $-1$).
\end{enumerate}
\end{lemma}

\begin{proof}
By construction of the join
\begin{equation}\label{E:isotropy of join}
\Iso_G(X_1* \cdots *X_n) = 
\bigcup_{\emptyset \neq I \subseteq \{1,\dots,n\}} \Iso_G(\prod_{i \in I} X_i) =
\bigcup_{\emptyset \neq I \subseteq \{1,\dots,n\}} \left\{ \bigcap_{i \in I} H_i : H_i \in \Iso_G(X_i) \right\}.
\end{equation}
%It is immediate from the construction of the join that $\Iso_G(X*Y)=\Iso_G(X) \cup \Iso_G(Y) \cup \Iso_G(X \times Y)$.
%As a result  
%\begin{equation}\label{E:isotropy of join}
%\Iso_G(X_1* \cdots *X_n) = \bigcup_{\emptyset \neq A \subseteq \{1,\dots,n\}} \left\{ \bigcap_{i \in A} H_i : H_i \in \Iso_G(X_i) \right\}.
%\end{equation}
Thus, $\Iso_G(S(U_{\leq 1})) \subseteq \Iso_G(S(U_{\leq 2})) \subseteq \cdots$.
Since $G$ is finite this sequence stabilizes on a collection $\F$.
It also follows that $\F=\Iso_G(S(V))$ for some representation with irreducible summands in $\Irr(U_\bullet)$.
Since hypothesis \ref{H:U} is in force, it is clear that $\F=\F(U_\bullet)$.

If $H \in \F$ then $\dim V^H \geq 1$ and moreover, if $K \gneq H$ then $\dim V^H > \dim V^K$.
Hypothesis \ref{H:U} implies that if $n$ is large enough then $U_{m \leq \bullet \leq n}$ contains $\oplus_{k+1} V$ and the lemma follows (inequality \ref{L:far in U:3} needs to be proven separately for $H \lneq G$ and $H=G$).
%One has to be careful to prove the inequality in \ref{L:far in U:3} separately for the case $H \lneq G$ and for the case $H=G$.
\end{proof}

\begin{proof}[Proof of Theorem \ref{T:linear spheres}]
{\em Proof of the stabilization:}
Fix $k \geq 0$.
By Lemma \ref{L:far in U} there exist integers $n_0 \geq m \geq 1$ such that $X \defeq S(U_{\leq m})$ and $Y' \defeq S(U_{m+1 \leq \bullet \leq n_0})$ satisfy
\[
\Iso_G(X)=\Iso_G(Y')=\F(U_\bullet)
\]
and for any $H \in \F(U_\bullet)$ 
\[
\dim X^H - \dim \bigcup_{K \gneq H} X^K \geq k+1 \qquad \text{and} \qquad
\dim Y'{}^H \geq k+1.
\]
Let $n \geq n_0$ and set $Y=S(U_{m+1 \leq \bullet \leq n})$ and $Z=S(U_{n+1})$.
Then, $S(U_{\leq n}) \cong X*Y$ and $S(U_{\leq n+1}) \cong X*Y*Z$ and the first statement of the theorem is that
\[
\map^G(X*Y,X*Y) \xto{f \mapsto f*\id_Z} \map^G(X*Y*Z, X*Y*Z)
\]
is a $k$-equivalence.
To prove this we apply Proposition \ref{P:stabilisation main result}.
First, $X,Y$ and $Z$ are finite $G$-CW complexes by Illman's result \cite{Illman78}.
Since $Y' \subseteq Y$, the choice of $m$ and $n_0$ guarantees that 
\[
\Iso(X)=\Iso(Y)=\Iso(X*Y*Z)=\F(U_\bullet)
\]
so hypothesis \ref{P:stab hypothesis 1 v2} of Proposition \ref{P:stabilisation main result} holds.

Hypothesis \ref{P:stab hypothesis 2 v2} also holds since $\dim Y^H \geq \dim Y'{}^H \geq k+1$ for all $H \in \F$, and our choice of $X$ satisfies hypothesis \ref{P:stab hypothesis 3 v2}.
Now, $(X*Y)^H \cong S(U_{\leq n})^H$ is itself a linear sphere of dimension $\dim X^H + \dim Y^H+1$, hence it is a $(\dim X^H + \dim Y^H)$-connected space.
Similarly $(X*Y*Z)^H$ is a linear sphere of dimension $\dim X^H+\dim (Y*Z)^H+1$ and hypothesis \ref{P:stab hypothesis 5 v2} also holds.

Now, $Z^H$ is a sphere and the join of a space $A$ with $S^m\cong S^0*\dots * S^0$ is homeomorphic to the $(m+1)$-fold unreduced suspension of $A$.
An iterated use of Proposition \ref{P:Freudenthal consequence new}\ref{P:Freudenthal b} %\ref{P:Freudenthal consequence} 
shows that the map $F(X^H,(X*Y)^H) \to F(X^H*Z^H,(X*Y)^H*Z^H)$, where $F(-,-)$ denotes the space of (non-equivariant) continuous maps, is a $(\dim (X*Y)^H -1)$-equivalence.
Hypothesis \ref{P:stab hypothesis 4 v2} of Proposition \ref{P:stabilisation main result} holds since $\dim (X*Y)^H-1=\dim X^H+\dim Y^H \geq \dim X^H+k+1$.

\noindent
{\em Calculation of the limit groups:}
Given $n \geq 1$ write $U=U_{\leq n}$ and $X=S(U) \subseteq U \subseteq S^U$.
Lemma \ref{L:far in U} guarantees that if $n$ is large enough then $\Iso_G(X)=\F(U_\bullet)$ and  $\dim X^H \geq k+2$ and $\dim X^H - \dim \cup_{K \gneq H} X^K \geq k+2$ for all $H \in \F(U_\bullet)$.
We will show that $\pi_i \map^G(X, X)$ are isomorphic to the groups in the statement of the theorem for $0 \leq i \leq k$.

This follows from Proposition \ref{P:stable term X} which we proceed to check its hypotheses \ref{P:stable term X:1}--\ref{P:stable term X:4}.
Clearly, $\dim X = \dim U -1$ which is hypothesis \ref{P:stable term X:1}.
Hypothesis \ref{P:stable term X:2} holds by the choice of $n$, and \ref{P:stable term X:3} since $S^U \setminus S(U)$ is $G$-equivalent to $S^0$. 
For hypothesis \ref{P:stable term X:4} let $\eta_{S(U)} \colon S(U) \to S^U$ be the map in \ref{V:eta} followed by the homeomorphism $\Sigma S(U) \cong S^U$.
By Proposition \ref{P:Freudenthal consequence new}\ref{P:Freudenthal a} $(\eta_{S(U)})^H=\eta_{S(U^H)}$ is a $(2\dim S(U)^H -2)$-equivalence for any $H \in \Iso_G(S(U))$, and by the choice of $U$ we have $2\dim S(U)^H -2 \geq \dim S(U)^H +k$.
\end{proof}

%\begin{void}\label{V:S^U and susp eta triangle}
% Recall that $S^U$ is the one point compactification of a representation $U$.
The ``distinguished'' fixed points $0,\infty \in S^U$ correspond, under the homeomorphism $S^U \cong \Sigma S(U)$, to $\alpha_0,\alpha_1 \in \Sigma S(U)$  in \ref{V:eta}. 
Write $A=\{\alpha_1\}$ and $B=\{\alpha_0,\alpha_1\}$.
Let 
\[
\map^G(S^U,S^U;\id_B) \qquad  \text{and} \qquad \map^G(S^U,S^U;\pi_B^A)
\]
be the spaces of maps $f \colon S^U \to S^U$ such that $f|_B=\id_B$ and, respectively, $f|_B=\pi_B^A \colon B \to A$.
By inspection of \ref{V:eta} the following triangle commutes
\begin{equation}\label{E:trangle susp eta}
\xymatrix{
\map^G(S(U),S(U)) \ar[rr]^(0.44){\susp \colon f \mapsto \Sigma f} \ar[rrd]_{(\eta_{S(U)})_*} & &
\map^G(\Sigma S(U), \Sigma S(U); \id_B) \ar[d]^{\cong} 
\\
& &
\map^G(S(U),\PP_{\alpha_0,\alpha_1} \Sigma S(U)). 
}
\end{equation}
%\end{void}

\begin{lemma}\label{L:susp and eta triangle}
Let $k \geq 0$.
With the notation above, if $\dim S(U)^H \geq k+2$ for every $H \in \Iso_G(S(U))$ then the diagonal arrow in \eqref{E:trangle susp eta} is a $k$-equivalence.
\end{lemma}

\begin{proof}
By Proposition \ref{P:Freudenthal consequence new}\ref{P:Freudenthal a} $(\eta_{S(U)})^H=\eta_{S(U^H)}$ is a $(2\dim S(U)^H -2)$-equivalence for any $H \in \Iso_G(S(U))$.
The result follows by applying  Lemma \ref{L:k equivalence on fibres} with $\emptyset \subseteq S(U)$  since $2\dim S(U)^H -2 \geq \dim S(U)^H +k$.
\end{proof}

\begin{proof}[Proof of Proposition \ref{P:map S(U) to S^U}]
Combine Lemma \ref{L:susp and eta triangle} with parts \ref{L:far in U:1} and \ref{L:far in U:2} of Lemma \ref{L:far in U}.
\end{proof}

% Recall that  $S^U=U \cup \{\infty\}$.
There is a continuous function 
\[
S^U \times S^V \xto{(u,v) \mapsto u+v} S^{U+V} \qquad (\text{where } u+\infty=\infty+v=\infty).
\] 
It gives rise to a homeomorphism $S^U \wedge S^V \cong S^{U+V}$.
Moreover, $(-) \wedge \alpha_0$ carries $B \subseteq S^U$ to $B\subseteq S^{U+V}$.

Under the homeomorphism $\Sigma S(U) \cong S^U$  the map $\susp$ in the square below has the form $\susp(f)(u) = |u| \cdot f(u/|u|)$ and $J_{S(V)}(f)(u+v) = |u|\cdot f(u/|u|)+v$ and $(f \wedge S^V)(u+v) = f(u)+v$.
The square commutes by inspection.
\begin{equation}\label{E:join S(V) vs wedge S^V}
\xymatrix{
\map^G(S(U),S(U)) \ar[d]_{J_{S(V)} \colon f \mapsto f*S(V)} \ar[r]^{\susp} &
\map^G(S^U,S^U;\id_B) \ar[d]^{f \mapsto f \wedge S^V} 
\\
\map^G(S(U+V),S(U+V)) \ar[r]_{\susp} & 
\map^G(S^{U+V},S^{U+V};\id_B).
}
\end{equation}

\begin{proof}[Proof of Theorem \ref{T:tom Dieck extended}]
Given $n \geq 1$ we will write $U=U_{\leq n}$ and $V=U_{n+1}$.
% As in \ref{V:S^U and susp eta triangle} set $A=\{\alpha_1 \}$ and $B=\{\alpha_0,\alpha_1\}$ where $\alpha_1 \in S^U$ is the basepoint.
Lemma \ref{L:susp and eta triangle} and Theorem \ref{T:linear spheres} imply that if $n$ is large enough then the arrows $\susp$ and $J_{S(V)}$ in \eqref{E:join S(V) vs wedge S^V} are $k$-equivalences, and therefore the arrow $(-)\wedge S^V$ on the right is a $k$-equivalence.

The inclusions $A \subseteq B \subseteq S^U$ %, where $\alpha_1$ is chosen as the basepoint, 
give rise to the following fibrations, with the inclusion maps as basepoints,
%\[
%\xymatrix@1{
%\map^G(\Sigma S(U),\Sigma S(U)) \ar@{->>}[r] & 
%\map^G(\{\alpha_0,\alpha_1\},\Sigma S(U)) \ar@{->>}[r] &
%\map^G(\{\alpha_1\},\Sigma S(U)).
%}
%\]
\[
\xymatrix@1{
\map^G(S^U,S^U) \ar@{->>}[r] & 
\map^G(B,S^U) \ar@{->>}[r] &
\map^G(A, S^U)
}
\]
and similar ones for $S^{U+V}$.
The fibres of fibrations  $\xymatrix@1{X_2 \ar@{->>}[r] & X_1 \ar@{->>}[r] & X_0}$ with basepoints $x_0,x_1,x_2$ fit into a fibre sequence $\xymatrix{F_{12} \ar@{^(->}[r] & F_{02} \ar@{->>}[r] & F_{01}}$.
We obtain a commutative diagram
\[
\xymatrix{
\map^G(S^U,S^U;\id_B) \ar@{^(->}[r] \ar[d]_{(-) \wedge S^V} &
\map^G_*(S^U,S^U) \ar@{->>}[r]^{\ev_{\alpha_0}} \ar[d]_{(-) \wedge S^V} & 
(S^U)^G \ar[d]_{(-) \wedge S^V}
\\
\map^G(S^{U+V},S^{U+V};\id_B) \ar@{^(->}[r] &
\map^G_*(S^{U+V},S^{U+V}) \ar@{->>}[r]_(0.6){\ev_{\alpha_0}} &
(S^{U+V})^G.
}
\]
%\[
%\xymatrix@1{
%\map^G(\Sigma S(U),\Sigma S(U);\alpha_0,\alpha)  \ar@{^(->}[r] &
%\map^G_*(\Sigma S(U),\Sigma S(U)) \ar@{->>}[r]^(0.65){\ev_{\alpha_0}} &
%\Sigma S(U)^G.
%}
%\]
Assume first that $\Irr(U_\bullet)$ contains the trivial representation.
By Lemma \ref{L:far in U}, if $n$ is large enough then $G \in \F(U_\bullet)=\Iso_G(U)$ and $\dim (S^U)^G \geq  k+2$.
Therefore the inclusion of the fibres in the square above are $k$-equivalences, hence $(-)\wedge S^V$ in the 2nd column are ones too.
Together with Theorem \ref{T:linear spheres} and since $\susp$ in \eqref{E:join S(V) vs wedge S^V} is a $k$-equivalence, the result follows.

If $\Irr(U_\bullet)$ does not contain the trivial representation then $G \notin \F(U_\bullet)$ and therefore $(S^U)^G \cong (S^{U+V})^G=\{\alpha_0,\alpha_1\}$.
Then $\map^G_*(S^U,S^U)$ is the disjoint union of the fibres of $\ev_{\alpha_0}$ over $\alpha_0$ and over $\alpha_1$, i.e $\map^G(S^U,S^U;\id_B)$ and $\map^G(S^U,S^U;\pi_B^A)$ respectively.
Similarly  $\map^G_*(S^{U+V},S^{U+V})$ has such decomposition.
We have seen that $(-)\wedge S^V$ induces a $k$-equivalence on the components over $\id_B$.
It induces a $k$-equivalence on the components over $\pi_B^A$ by Lemma \ref{L:tau wedge S^V homotopy commutative} below.
The result follows from Theorem \ref{T:linear spheres} and since $\susp$ in \eqref{E:join S(V) vs wedge S^V} is a $k$-equivalence.
\end{proof}

%\begin{void}\label{V:curlyvee}
Consider $G$-spaces $(X;x_0,x_1)$ with distinguished points $x_0,x_1 \in X^G$.
Examples are given by $(\Sigma X;\alpha_0,\alpha_1)$, see \ref{V:eta}.  
Given $(X;x_0,x_1)$ and $(Y;y_0,y_1)$ let 
\[
X \curlyvee Y  \ \defeq \  \left(X \coprod Y\right) \,/\,x_1 \sim y_0.
\]
%be the disjoint union of $X$ and $Y$ with $x_1$ identified with $y_0$.
%We will write $[t,x]$ for the equivalence classes of the points of $\Sigma X$ where $\alpha_0=[0,x]$ and $\alpha_1=[1,x]$, see \ref{V:eta}, are the distinguished fixed points.
Let $[t,x]$ denote the equivalence classes of the points of $\Sigma X$. % and $\alpha_0=[0,x]$ and $\alpha_1=[1,x]$, see 
There is a pinch map $\Sigma X \to \Sigma X \curlyvee \Sigma X$ where $[t,x] \mapsto [2t,x]$ if $0 \leq t \leq \half$ and $[t,x] \mapsto [2t-1,x]$ if $\half \leq t \leq 1$.
For any $G$-space $Z$ there results
\[
\map^G(\Sigma X \curlyvee \Sigma X,Z) \xto{\pinch^*} \map^G(\Sigma X,Z).
\]
The space on the left is identified with the space of pairs $(f,g)$ of maps $\Sigma X \to Z$ such that $f(\alpha_1)=g(\alpha_0)$.
We will denote $\pinch^*(f,g)$ by $f+g$.

Let $\inv \colon \Sigma X \to \Sigma X$ be the map $[t,x] \mapsto [1-t,x]$.
Let $A \subseteq B \subseteq \Sigma X$ denote $\{\alpha_1\}$ and $\{\alpha_0\,\alpha_1\}$ respectively.
Let $\map^G(\Sigma X,\Sigma X;\id_B)$ be the space of maps with $f|_B = \id_B$.
Similarly $\map^G(\Sigma X,\Sigma X;\pi_B^A)$ is the space of maps with $f|_B=\pi_B^A \colon B \to A$.
Define maps
\begin{eqnarray}
&& \sigma \colon \map^G(\Sigma X,\Sigma X;\id_B) \xto{f \mapsto \inv+f} \map^G(\Sigma X,\Sigma X;\pi_B^A) 
\label{D:tau and sigma equivalences}
\\
\nonumber
&& \tau \colon \map^G(\Sigma X,\Sigma X;\pi_B^A) \xto{f \mapsto \id_{\Sigma X}+f} \map^G(\Sigma X,\Sigma X;\id_B) 
\end{eqnarray}
%\end{void}

\begin{lemma}\label{L:tau sigma he}
The maps $\tau$ and $\sigma$ in \eqref{D:tau and sigma equivalences} are equivariant homotopy equivalences.
\end{lemma}

\begin{proof}
First, $(\id+\inv) \colon \Sigma X \to \Sigma X$ is equivariantly homotopic to the constant map $\alpha_0$ (via the homotopy $h_s([t,x])=[2ts,x]$ if $0 \leq t \leq \frac{1}{2}$ and $h_s([t,x])=[s(2-2t),x]$ if $\frac{1}{2} \leq t \leq 1$).
Similarly $(\inv+\id)$ is homotopic to the constant map $\alpha_1$.

Given $f \colon \Sigma X \to \Sigma X$ such that $f(\alpha_0)=\alpha_0$, there is a natural homotopy $\alpha_0+f \simeq f$.
Similarly, if $f(\alpha_0)=\alpha_1$ there is a natural homotopy $\alpha_1+f \simeq f$.
 
There is a natural homotopy between the maps
\begin{eqnarray*}
&& \Sigma X \xto{\pinch} \Sigma X \curlyvee \Sigma X \xto{\Sigma X \curlyvee \pinch} \Sigma X \curlyvee \Sigma X \curlyvee \Sigma X  \text{ and } \\
&& \Sigma X \xto{\pinch} \Sigma X \curlyvee \Sigma X \xto{\pinch \curlyvee \Sigma X} \Sigma X \curlyvee \Sigma X \curlyvee \Sigma X.
\end{eqnarray*}
Thus, there are homotopies $\tau(\sigma(f))=\id+(\inv+f) \simeq (\id+\inv)+f \simeq \alpha_0 + f \simeq f$ and $\sigma(\tau(f)) = \inv+(\id+f) \simeq (\inv+\id)+f \simeq \alpha_1+f \simeq f$ natural in $f$.
\end{proof}

%The topology on $S^U$ is Euclidean on $U$ and neighbourhoods of $\infty \in S^U$ are the open balls $\Ball^{S^U}(\infty,R)=\{u\in U : |u|>R\} \cup \{\infty\}$.
Neighbourhoods of $\infty \in S^U$ are the open balls $\{u\in U : |u|>R\} \cup \{\infty\}$.

\begin{lemma}\label{L:tau wedge S^V homotopy commutative}
Let $U, V$ be (orthogonal) representations of $G$.
There is a homotopy commutative square in which the horizontal maps are homotopy equivalences
\[
\xymatrix{
\map^G(S^U,S^U;\pi_B^A) \ar[r]^{\tau} \ar[d]_{(-)\wedge S^V} &
\map^G(S^U,S^U;\id_B) \ar[d]^{(-)\wedge S^V}
\\
\map^G(S^{U+V},S^{U+V};\pi_B^A) \ar[r]^\tau &
\map^G(S^{U+V},S^{U+V};\id_B).
}
\]
\end{lemma}

\begin{proof}
Since $\Sigma S(U) \cong S^U$ the maps $\tau$ %defined in \eqref{D:tau and sigma equivalences} and are 
in \eqref{D:tau and sigma equivalences} are
homotopy equivalences by Lemma \ref{L:tau sigma he}.

Fix once and for all a homeomorphism $\vp \colon [0,1] \to [0,\infty]$ with $\vp(0)=0$.
Let $U \oplus V$ be the orthogonal sum (thus, $|u+v|=|u|+|v|$).
We model the pinch map $S^U \to S^U \curlyvee S^U$ by
\[
\pinch(u) = \left\{
\begin{array}{ll}
\vp(|u|)u       & \text{if } |u| \leq 1 \text{      (note: $\infty \cdot u=\infty$ for $|u| =1$)} \\
(1-\frac{1}{|u|})u  & \text{if } |u| \geq 1 \text{ or } u=\infty.
\end{array}\right.
\] 
%>>>>>>>>>
\iffalse
For any $f \in \map^G(S^U,S^U;\pi_B^A)$, inspection of the definitions shows that $\tau(f) \wedge S^V$ and $\tau(f \wedge S^V)$ are given by the formulas
\[
u+v \mapsto 
\left\{\begin{array}{ll}
\vp(|u|)u+v & \text{if } |u| \leq 1 \\
f((1-\frac{1}{|u|})u)+v & \text{if } 1 \leq |u| <\infty \\
\infty & \text{if } u=\infty
\end{array}\right.
\qquad
u+v \mapsto 
\left\{\begin{array}{ll}
\vp(|u+v|)u+\vp(|u+v|)v & \text{if } |u+v| \leq 1 \\
f((1-\frac{1}{|u+v|})u)+(1-\frac{1}{|u+v|}v & \text{if } 1 \leq |u+v| <\infty \\
\infty & \text{if } u+v=\infty
\end{array}\right.
\]
Our goal is to find a homotopy between them.
\fi
%<<<<<<<<<<
Define 
%\[
$h \colon I \times \map^G(S^U,S^U;\pi_B^A) \times S^{U+V} \to S^{U+V}$
%\]
where $A=\{\infty\}$ and $B=\{0,\infty\} \subseteq S^U$:
\[
h(t,f,u+v) = 
\left\{
\begin{array}{ll}
\vp(|u+tv|)(u+tv)+(1-t)v & \text{if }  u+v \neq \infty \text{ and } |u+tv| \leq 1 \\
f((1-\frac{1}{|u+tv|})u)+(1-\frac{t}{|u+tv|})v   & \text{if }  u+v \neq \infty \text{ and } |u+tv| \geq 1 \\
\infty & \text{if } u+v=\infty
\end{array}\right.
\]
This is well defined because if $|u+tv|=1$ then $\infty \cdot (u+tv)=\infty$ and $f(0)=\pi_B^A(0)=\infty$.
In what follows we will show that $h$ is continuous.
Once this is established, the adjoint of $h$ gives a map $H \colon I \times \map^G(S^U,S^U;\pi_B^A) \to \map^G(S^{U+V},S^{U+V};\id_B)$ since $h(t,f,0)=0$ and $h(t,f,\infty)=\infty$.
It is a homotopy from $(-) \wedge S^V \circ \tau$ to $\tau \circ (-) \wedge S^V$ in the square above, which completes the proof.
To show $h$ is continuous we apply the pasting lemma to the following subsets of $I \times \map^G(S^U,S^U;\pi_B^A) \times S^{U+V}$
\begin{eqnarray*}
&& D=\{ (t,f,u+v) : \text{ either } u+v \neq \infty \text{ and } |u+tv| \leq 1 \text{ or } t=0 \text{ and } u+v=\infty \} \\
&& E=\{(t,f,u+v) : \text{ either } u+v\neq\infty \text{ and } |u+tv| \geq 1 \text{ or } u+v=\infty \}.
\end{eqnarray*}

\noindent
{\em Claim 1:} $D$ and $E$ are is a closed subsets of $I \times \map^G(S^U,S^U;\pi_B^A) \times S^{U+V}$.

\noindent
Proof: 
We replace $D$ and $E$ with their images in $I \times S^{U+V}$ under the projection.
%Choose $x=(t_0,u_0+v_0) \notin D$.
%If $u_0+v_0 \neq \infty$ then $x \in \{ (t,u+v) : |u+tv|>1\}$ which is an open subset of $I \times (U\oplus V)$ and is disjoint from $D$.
%%\comment{I identify $D$ and $E$ with their projections into $I \times S^{U+V}$}
%If $u_0+v_0=\infty$ then $t_0>0$ and $x$ is in the preimage of $(1,\infty]$ of the continuous map $(\frac{t_0}{2},1] \times S^{U+V} \xto{(t,u+v) \mapsto |u+tv|} [0,\infty]$ which is disjoint from $D$.
The complement of $E$ is $\{ (t,u+v) : u+v \neq \infty \text{ and } |u+tv|<1\}$ which is clearly open.
The complement of $D$ is the preimage of $(1,\infty]$ under the map $\lambda \colon I \times S^{U+V} \setminus \{(0,\infty)\} \xto{(t,u+v) \mapsto |u+tv|} [0,\infty]$ where $\lambda(t,\infty)=\infty$.
It is clearly continuous at any $(t_0,u_0+v_0)$ with $u_0+v_0 \neq \infty$;
It is continuous at points $(t_0,\infty)$ with $t_0>0$ since given $R>0$, whenever $(t,u+v) \in I \times S^{U+V} \setminus \{(0,\infty)\}$ is such that $t>\frac{t_0}{2}$ and $|u+v|>R+\frac{2R}{t_0}$ we have either $|u|>R$ in which case $|\lambda(t,u+v)|\geq |u|>R$ or $|v|>\frac{2R}{t_0}$ so $|\lambda(t,u+v)| \geq t|v| > R$. 
%The complement of $D$ is conatined in $(0,1]\times S^{U+V}$ and is the preimage of $(1,\infty]$ under the continuous map $(0,1]\times S^{U+V} \xto{(t,u+v) \mapsto |u+tv|} [0,\infty]$ where $u+tv=\infty$ if $u+v=\infty$.
%\hfill QED.

\noindent
Claim 2: $h$ is continuous on $D$.

\noindent
Proof: 
We may replace $D$ with its projection $D'$ in $I \times S^{U+V}$ and prove that $\lambda \colon D' \to S^{U+V}$ defined by $\lambda(t,u+v)=\vp(|u+tv|)u + (t\vp(|u+tv|)+1-t)v$ if $u+v \neq \infty$ and $\lambda(0,\infty)=\infty$, is continuous.
Continuity is clear away from $(0,\infty) \in D'$.
Continuity at $(0,\infty)$ will follow once we show that given $R>1$, if $(t,u+v) \in D'$ is such that $|u+v|>R+2$ then $|\lambda(t,u+v)|>R$.
Indeed,  $|u+tv| \leq 1$ implies $|u|, |tv| \leq 1$, hence $|v|> R+1$ and $t\leq\frac{1}{|v|}$, and it follows that $|\lambda(t,u+v)| \geq (t\vp(|u+tv|)+1-t)|v| \geq (1-t)|v| \geq |v|-1 >R$.
\iffalse
Clearly $h$ is continuous on $\{(t,f,u+v) \in D : u+v \neq \infty\}$ as the composition of continuous functions.
To prove continuity at points $(0,f,\infty) \in D$ note first that $h(0,f,\infty)=\infty$.
If $R>0$ then $\O=\{ (t,g,u+v) \in D : R+2 < |u+v| \leq \infty\}$ is open in $D$.
Then $f(\O) \subseteq \Ball_{S^U}(\infty,R)$ because if $(t,g,u+v) \in \O$ then $|u+v| > R+2$ and $|u+tv|\leq 1$, hence $|v|>R+1$ and $t<\frac{1}{|v|}$ and therefore 
\[
|h(t,g,u+v)| \geq (t\vp(|u+tv|)+1-t)|v| \geq (1-t)|v| > \frac{R}{R+1} (R+1) =R.
\]
\fi
\iffalse
We check continuity at $(0,f,\infty) \in D$.
Note that $h(0,f,\infty)=\infty$.
Let $R>0$ and set $\O=\{ (t,g,u+v) \in D : R+2 < |u+v| \leq \infty\}$.
Then $\O$ is open in $D$ and we finish the proof by showing that $h(\O) \subseteq \Ball^{S^{U+V}}(\infty,R)$.
Indeed, $(t,g,u+v) \in \O$ implies $|u+v|>R+2$ and $|u+tv|\leq 1$.
Therefore $|u| \leq 1$ and $t|v| \leq 1$, whence $|v|>R+1$ and $t<\frac{1}{R+1}$.
It follows that
\[
|h(t,g,u+v)| \geq (t\vp(|u+tv|)+1-t)|v| \geq (1-t)|v| > \frac{R}{R+1} (R+1) =R.
\]
\fi

\noindent
Claim 3: $h$ is continuous on $E$.

\noindent
Proof:
Let $E'$ be the projection of $E$ in $I \times S^{U+V}$.
Define $\lambda \colon E' \to S^{U+V}$ by $\lambda (t,u+v) (1-\frac{1}{|u+tv|})u+(1-\frac{t}{|u+tv|})v$ if $u+v \neq \infty$ and $\lambda(t,\infty)=\infty$.
Then $h|_E(t,f,u+v)=(f\wedge S^V)(\lambda(t,u+v))$, and since $f \mapsto f \wedge S^V$ and the evaluation map are continuous, it remains to show that $\lambda$ is continuous.
This is clear away from the points $(t,\infty)$.
Continuity of $\lambda$ at points $(t_0,\infty) \in E'$ would follow once we show that given $R>1$, if $(t,u+v) \in E'$ is such that $|u+v|>R+2$, then $|\lambda(t,u+v)|>R$.
Indeed, $|u+tv| \geq |u|, t|v|$ and therefore $(1-\frac{1}{|u+tv|})|u| \geq |u|-1$ and $(1-\frac{t}{|u+tv|})|v| \geq |v|-1$ (take special care when $u=0$ or $v=0$ or $t=0$).
It follows that if $|u+v|>R+2$ then  $|\lambda(t,u+v)| \geq |u|-1+|v|-1 > R$.
\end{proof}

\bibliography{bibliography}{}
\bibliographystyle{plain}

\end{document}